\documentclass[11pt]{amsart}

\usepackage{amsmath, amssymb, amsthm, amscd}
\usepackage{url}
\usepackage{graphicx}
\usepackage[hidelinks,pagebackref,pdftex]{hyperref}
\usepackage{color}
\usepackage{import}
\usepackage[margin=1.2in]{geometry}

\usepackage{marginnote}
\long\def\@savemarbox#1#2{\global\setbox#1\vtop{\hsize\marginparwidth 
  \@parboxrestore\tiny\raggedright #2}}
\renewcommand*{\backref}[1]{}
\renewcommand*{\backrefalt}[4]{
  \ifcase #1
  [No citations.]
  \or [#2]
  \else [#2]
  \fi }

\AtBeginDocument{%
   \def\MR#1{}
}

\numberwithin{equation}{section}
\theoremstyle{plain}
\newtheorem{theorem}[equation]{Theorem}
\newtheorem{conj}[equation]{Conjecture}
\newtheorem{lemma}[equation]{Lemma}
\newtheorem{corol}[equation]{Corollary}
\newtheorem{corollary}[equation]{Corollary}
\newtheorem{proposition}[equation]{Proposition}

\newtheorem*{namedtheorem}{\theoremname}
\newcommand{\theoremname}{testing}

\theoremstyle{definition}
\newtheorem{definition}[equation]{Definition}
\newtheorem{remark}[equation]{Remark}
\newtheorem{example}[equation]{Example}

\newcommand{\from}{\colon} 

\newcommand{\ZZ}{{\mathbb{Z}}}

\newcommand{\vol}{\operatorname{vol}}
\newcommand{\area}{\operatorname{area}}
\newcommand{\guts}{\operatorname{guts}}
\newcommand{\bdy}{\partial}
\newcommand{\cut}{{\backslash \backslash}}
\renewcommand{\setminus}{{\smallsetminus}}

\newcommand{\tw}{\operatorname{tw}}
\newcommand{\Cr}{\operatorname{cr}}

\newcommand{\refthm}[1]{Theorem~\ref{Thm:#1}}
\newcommand{\reflem}[1]{Lemma~\ref{Lem:#1}}
\newcommand{\refprop}[1]{Proposition~\ref{Prop:#1}}
\newcommand{\refcor}[1]{Corollary~\ref{Cor:#1}}
\newcommand{\refrem}[1]{Remark~\ref{Rem:#1}}

\newcommand{\refeqn}[1]{\eqref{Eqn:#1}}
\newcommand{\refitm}[1]{\eqref{Itm:#1}}
\newcommand{\refdef}[1]{Definition~\ref{Def:#1}}
\newcommand{\refsec}[1]{Section~\ref{Sec:#1}}
\newcommand{\reffig}[1]{Figure~\ref{Fig:#1}}
\newcommand{\refexa}[1]{Example~\ref{Exa:#1}}

\title{Geometry of alternating links on surfaces}
\author{Joshua A. Howie}
\address{School of Mathematical Sciences, Monash University, VIC 3800, Australia}
\email{josh.howie@monash.edu}
\author{Jessica S. Purcell}
\address{School of Mathematical Sciences, Monash University, VIC 3800, Australia}
\email{jessica.purcell@monash.edu}

\begin{document}

\begin{abstract}
We consider links that are alternating on surfaces embedded in a compact 3-manifold. We show that under mild restrictions, the complement of the link decomposes into simpler pieces, generalising the polyhedral decomposition of alternating links of Menasco. We use this to prove various facts about the hyperbolic geometry of generalisations of alternating links, including weakly generalised alternating links described by the first author. We give diagrammatical properties that determine when such links are hyperbolic, find the geometry of their checkerboard surfaces, bound volume, and exclude exceptional Dehn fillings. 
\end{abstract}

\maketitle

\section{Introduction}\label{Sec:Intro}

In 1982, Thurston proved that all knots in the 3-sphere are either satellite knots, torus knots, or hyperbolic \cite{thu82}. Of all classes of knots, alternating knots (and links) have been the most amenable to the study of hyperbolic geometry in the ensuing years. Menasco proved that aside from $(2,q)$-torus knots and links, all prime alternating links are hyperbolic \cite{men84}. Lackenby bounded their volumes in terms of a reduced alternating diagram \cite{lac04}. Adams showed that checkerboard surfaces in hyperbolic alternating links are quasifuchsian \cite{ada07}; see also \cite{fkp14}. Lackenby and Purcell bounded cusp geometry \cite{lp16}. Additionally, there are several open conjectures on relationships between their geometry and an alternating diagram that arise from computer calculation, for example that of Thistlethwaite and Tsvietkova \cite{tt14}.

Because hyperbolic geometric properties of alternating knots can be read off of an alternating diagram, it makes sense to try to generalise these knots, and to try to distill the properties of such knots that lead to geometric information. One important property is the existence of a pair of essential checkerboard surfaces, which are known to characterise an alternating link \cite{gre17, how17}. These lead to a decomposition of the link complement into polyhedra, which can be given an angle structure \cite{lac00}. This gives tools to study surfaces embedded in the link complement. Angled polyhedra were generalised to ``angled blocks'' by Futer and Gu\'eritaud to study arborescent and Montesinos links \cite{fg09}. However, their angled blocks do not allow certain combinatorial behaviours, such as bigon faces, that arise in practice. Here, we generalise further, to angled decompositions we call angled chunks. We also allow decompositions of manifolds with boundary.

We apply these techniques to broad generalisations of alternating knots in compact 3-manifolds, including generalisations due to the first author \cite{how15t}. An alternating knot has a diagram that is alternating on a plane of projection $S^2\subset S^3$. We may also consider alternating projections of knots onto higher genus surfaces in more general 3-manifolds. Adams was one of the first to consider such knots; in 1994 he studied alternating projections on a Heegaard torus in $S^3$ and lens spaces, and their geometry \cite{ada94}. The case of higher genus surfaces in $S^3$ has been studied by Hayashi \cite{hay95} and Ozawa \cite{oza06}. By generalising further, Howie \cite{how15t} and Howie and Rubinstein \cite{hr16} obtain a more general class of alternating diagrams on surfaces in $S^3$ for which the checkerboard surfaces are guaranteed to be essential. Here, we obtain similar results without restricting to $S^3$. 

In this paper, we utilise essential surfaces and angled decompositions to prove a large number of results on the geometry of classes of generalisations of alternating knots in compact 3-manifolds. We identify from a diagram conditions that guarantee such links are hyperbolic, satellite, or torus links, generalising \cite{men84, hay95}. We identify the geometry of checkerboard surfaces, either accidental, quasifuchsian, or a virtual fiber, generalising \cite{fkp14, ada07}. We also bound the volumes of such links from below, generalising \cite{lac04}, and we determine conditions that guarantee their Dehn fillings are hyperbolic, generalising \cite{lac00}. We re-frame all these disparate results as consequences of the existence of essential surfaces and an angled decomposition. It is likely that much of this work will apply to additional classes of link complements and 3-manifolds with such a decomposition, but we focus here on alternating links. 

To state our results carefully, we must describe the alternating links on surfaces carefully, and that requires ruling out certain trivial diagrams and generalisations of connected sums. Additionally, the link must project to the surface of the diagram in a substantial way, best described by a condition on representativity $r(\pi(L),F)$, which is adapted from a notion in graph theory. These conditions are very natural, described in detail in \refsec{Gen}, where we define a \emph{weakly generalised alternating link}, \refdef{WeaklyGeneralisedAlternating}. One consequence is the following.

\begin{theorem}\label{Thm:Intro}
Let $\pi(L)$ be a weakly generalised alternating projection of a link $L$ onto a generalised projection surface $F$ in a 3-manifold $Y$. Suppose $Y$ is compact, orientable, irreducible, and if $\bdy Y \neq \emptyset$, then $\bdy Y$ is incompressible in $Y\setminus N(F)$. 
Finally, suppose $Y\setminus N(F)$ is atoroidal and contains no essential annuli with both boundary components on $\bdy Y$.
If $F$ has genus at least one, and the regions in the complement of $\pi(L)$ on $F$ are disks, and the representativity $r(\pi(L),F)>4$, then
\begin{enumerate}
\item $Y\setminus L$ is hyperbolic.
\item $Y\setminus L$ admits two checkerboard surfaces that are essential and quasifuchsian.
\item The hyperbolic volume of $Y\setminus L$ is bounded below by a function of the twist number of $\pi(L)$ and the Euler characteristic of $F$:
\[ \vol(Y\setminus L)\geq\frac{v_8}{2}(\tw(\pi(L))-\chi(F)-\chi(\bdy Y)).\]
Here, in the case that $Y\setminus N(L)$ has boundary components of genus greater than one, we take $\vol(Y\setminus L)$ to mean the volume of the unique hyperbolic manifold with interior homeomorphic to $Y\setminus L$, and with higher genus boundary components that are totally geodesic.
\item Further, if $L$ is a knot with twist number greater than eight, or the genus of $F$ is at least five, then all non-trivial Dehn fillings of $Y\setminus L$ are hyperbolic.
\end{enumerate}
\end{theorem}

\refthm{Intro} follows from Theorems~\ref{Thm:hress}, \ref{Thm:Hyperbolic}, \ref{Thm:volguts}, and~\ref{Thm:DehnFilling} and Corollaries~\ref{Cor:Quasifuchsian} and~\ref{Cor:DehnGenusBound} in this paper. More general results also hold, stated below.

In the classical setting, $Y=S^3$. 
The results of \refthm{Intro} immediately apply to large classes of the generalisations of alternating knots in $S^3$ studied in \cite{ada94, oza06, hay95, how15t, hr16}. However, since we allow more general $Y$, \refthm{Intro} also applies more broadly, for example to all cellular alternating links in the thickened torus $T^2\times I$, as in \cite{ckp}. Geometric properties of alternating links in $T^2\times I$ arising from Euclidean tilings have also been studied recently in \cite{acm17} and \cite{ckp}; see \refexa{T2xI}.

\subsection{Organisation of results}

In \refsec{Gen}, we define our generalisations of alternating knots, particularly conditions that ensure our generalisations are nontrivial with reduced diagrams. \refsec{Chunks} introduces angled chunks, and proves that complements of these links have an angled chunk decomposition. This gives us tools to discuss the hyperbolicity of their complements in \refsec{Hyperbolic}. The techniques can be applied to identify geometry of surfaces in \refsec{Accidental}, and to bound volumes in \refsec{Volume}. Finally, we restrict their exceptional Dehn fillings in \refsec{Filling}. 

\subsection{Acknowledgements}
Both authors were partially supported by grants from the Australian Research Council.
The first author was also supported by a Lift-off Fellowship from the Australian Mathematical Society. 
We thank Colin Adams, Abhijit Champanerkar, Effie Kalfagianni, Ilya Kofman, Makoto Ozawa, and Hyam Rubinstein for helpful conversations. We also thank the anonymous referee for their suggestions.

\section{Generalisations of alternating knots and links}\label{Sec:Gen}
An alternating knot has a diagram that is alternating on a plane of projection, $S^2$ embedded in $S^3$. One may generalise to an alternating diagram on another surface embedded in a different 3-manifold. A more general definition is the following.

\begin{definition}\label{Def:ProjSfce}
Let $F_i$ be a closed orientable surface for $i=1,\ldots,p$, and let $Y$ be a compact orientable irreducible 3-manifold. A \emph{generalised projection surface} $F$ is a piecewise linear embedding,
$F\from \bigsqcup_{i=1}^{p}F_i\hookrightarrow Y$
such that $F$ is non-split in $Y$.
(Recall a collection of surfaces $F$ is \emph{non-split} if every embedded $2$-sphere in $Y\setminus F$ bounds a $3$-ball in $Y\setminus F$.)
Since $F$ is an embedding we will also denote the image of $F$ in $Y$ by $F$.
\end{definition}

Note that our definitions ensure that if $F$ is a generalised projection surface and some component $F_i$ is a 2-sphere, then $F$ is homeomorphic to $S^2$, and $Y$ is homeomorphic to $S^3$. 

\begin{definition}\label{Def:GenDiagram}
For $F$ a generalised projection surface, a link $L\subset F\times I\subset Y$ can be projected onto $F$ by $\pi \from F\times I \to F$.
We call $\pi(L)$ a \emph{generalised diagram}.
\end{definition}

Note that every knot has a trivial generalised diagram on the torus boundary of a regular neighbourhood of the knot. Such diagrams are not useful. To ensure nontriviality, we require conditions expressed in terms of representativity, adapted from graph theory:

\begin{definition}\label{Def:Representativity}
  Let $F$ be a closed orientable (possibly disconnected) surface embedded in a compact orientable irreducible 3-manifold $Y$. A regular neighbourhood $N(F)$ of $F$ in $Y$ has the form $F\times (-1,1)$. Because every component $F_i$ of $F$ is 2-sided, $Y\setminus N(F)$ has two boundary components homeomorphic to $F_i$, namely $F_i\times\{-1\}$ and $F_i\times\{1\}$. Let $F_i^-$ denote the boundary component coming from $F_i\times\{-1\}$ and $F_i^+$ the one from $F_i\times\{1\}$.
Let $\pi(L)$ be a link projection onto $F$. If $\ell$ is an essential curve on $F$, isotope $\ell$ so that it meets $\pi(L)$ transversely in edges, not crossings.

\begin{itemize}
\item The \emph{edge-representativity} $e(\pi(L),F)$ is the minimum number of intersections between $\pi(L)$ and any essential curve $\ell\subset F$.

\item Define $r^-(\pi(L),F_i)$ to be the minimum number of intersections between the projection of $\pi(L)$ onto $F_i^-$ and the boundary of any compressing disk for $F_i^-$ in $Y\setminus F$. If there are no compressing disks for $F_i^-$ in $Y\setminus F$, then set $r^-(\pi(L),F_i)=\infty$. Define $r^+(\pi(L),F_i)$ similarly, using $F_i^+$.

\item The \emph{representativity} $r(\pi(L),F)$ is the minimum of \[\bigcup_i (r^-(\pi(L),F_i)\cup r^+(\pi(L),F_i)).\] 

\item Also, define $\hat{r}(\pi(L),F)$ to be the minimum of \[\bigcup_i \max (r^-(\pi(L),F_i), r^+(\pi(L),F_i)).\]
\end{itemize}
\end{definition}

Note that if $F$ has just one component, the representativity $r(\pi(L),F)$ counts the minimum number of times the boundary of any compression disk for $F$ meets the diagram $\pi(L)$. As for $\hat{r}(\pi(L),F)$, there may be a compression disk on one side of $F$ whose boundary meets $\pi(L)$ less than $\hat{r}(\pi(L),F)$ times, but all compression disks on the opposite side have boundary meeting the diagram at least $\hat{r}(\pi(L),F)$ times. 
We will require diagrams to have representativity at least $2$, $4$, or more, depending on the result.

\begin{example}
  Let $Y$ be the thickened torus $Y=T^2\times [-1,1]$, and let $F$ be the torus $T^2\times\{0\}$. Consider the generalised diagram $\pi(L)$ shown in \reffig{CheckColourable}. The edge-representativity $e(\pi(L), F)$ is zero for this example, since the curve of slope $0/1$ (the horizontal edge of the rectangle shown in the figure) does not meet $\pi(L)$. However, since there are no compressing disks for $F$ in $Y$, both $r^-(\pi(L),F)$ and $r^+(\pi(L),F)$ are infinite, and thus so are $r(\pi(L),F)$ and $\hat{r}(\pi(L),F)$. 
\end{example}

\begin{figure}
\includegraphics{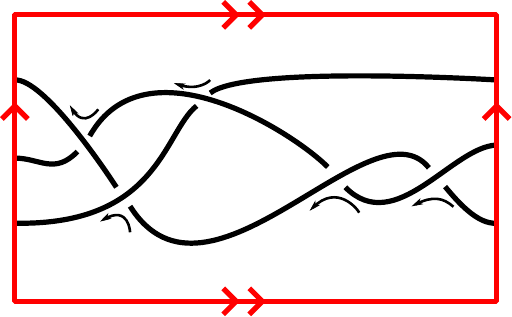}
  \caption{An example of an alternating diagram on a torus that is not checkerboard colourable.}
  \label{Fig:CheckColourable}
\end{figure}

We would also like to consider diagrams that are appropriately reduced; for example, there should be no nugatory crossings on $F$. In the alternating case on $S^2\subset S^3$, the condition is that the diagram be \emph{prime}: if an essential curve intersects the diagram exactly twice then it bounds a region of the diagram containing a single embedded arc. 

There are two natural generalisations of this condition. First, we can define a generalised diagram to be \emph{strongly prime} if whenever a loop $\ell\subset F$ intersects $\pi(L)$ exactly twice, then $\ell$ bounds a disk $D\subset F$ such that $\pi(L)\cap D$ is a single embedded arc. This notion was considered by Ozawa~\cite{oza06}. However, we will consider a weaker notion:

\begin{definition}\label{Def:WeaklyPrime}
A generalised diagram $\pi(L)$ on generalised projection surface $F$ is \emph{weakly prime} if whenever $D\subset F_i\subset F$ is a disk in a component $F_i$ of $F$ with $\bdy D$ intersecting $\pi(L)$ transversely exactly twice, the following holds:
\begin{itemize}
\item If $F_i$ has positive genus, then the intersection $\pi(L)\cap D$ is a single embedded arc.
\item If $F_i \cong S^2$, then $\pi(L)$ has at least two crossings on $F_i$ and either the intersection $\pi(L)\cap D$ is a single embedded arc, or $\pi(L)\cap (F_i\setminus D)$ is a single embedded arc.
\end{itemize}
\end{definition}

If $F\cong S^2$, so $\pi(L)$ is a diagram in the usual sense, then a weakly prime diagram is equivalent to a reduced prime diagram.
Additionaly, strongly prime implies weakly prime, but the converse is not true: If an essential curve meets $\pi(L)$ exactly twice but does not bound a disk in $F$, the generalised diagram is not strongly prime, but could be weakly prime. 

The generalised diagrams we consider will be alternating. Because the diagrams might be disconnected, we need to define carefully an alternating diagram in this case. 

\begin{definition}\label{Def:Alternating}
  A generalised diagram $\pi(L)$ is said to be \emph{alternating} if for each region of $F\setminus\pi(L)$, each boundary component of the region is alternating. That is, each boundary component of each region of $F\setminus\pi(L)$ can be given an orientation such that crossings run from under to over in the direction of orientation. 
\end{definition}

\reffig{CheckColourable} shows an example of an alternating diagram on a torus. In this diagram, there is exactly one region that is not a disk; it is an annulus. Each boundary component of the annulus can be oriented to be alternating, so this diagram satisfies \refdef{Alternating}. However, note that there is no way to orient the annulus region consistently to ensure that the induced orientations on the boundary are alternating (in the same direction: under to over).

\begin{definition}\label{Def:CheckColourable}
  A generalised diagram $\pi(L)$ is said to be \emph{checkerboard colourable} if each region of $F\setminus\pi(L)$ can be oriented such that the induced orientation on each boundary component is alternating: crossings run from under to over in the direction of orientation. Given a checkerboard colourable diagram, regions on opposite sides of an edge of $\pi(L)$ will have opposite orientations. We colour all regions with one orientation white and the other shaded; this gives the checkerboard colouring of the definition.
\end{definition}

The diagram of \reffig{CheckColourable} is not checkerboard colourable. 

Typically, we will consider the following general class of links. 

\begin{definition}[Reduced alternating on $F$]\label{Def:AltKnots}
Let $Y$ be a compact, orientable, irreducible 3-manifold with generalised projection surface $F$ such that if $\bdy Y\neq \emptyset$, then $\bdy Y$ is incompressible in $Y\setminus N(F)$.
A generalised diagram $\pi(L)$ on $F$ of a knot or link $L$ is \emph{reduced alternating} if 
  \begin{enumerate}
  \item\label{Itm:Alternating} $\pi(L)$ is alternating on $F$,
  \item $\pi(L)$ is weakly prime,
  \item\label{Itm:Connected} $\pi(L)\cap F_i \neq \emptyset$ for each $i=1, \dots, p$, and
  \item\label{Itm:slope} each component of $L$ projects to at least one crossing in $\pi(L)$.
  \end{enumerate}
\end{definition}

Knots and links that satisfy \refdef{AltKnots} have been studied in various places. Such links on a surface $F$ in $S^3$ that are checkerboard colourable with representativity $r(\pi(L),F) \geq 4$ are called \emph{weakly generalised alternating links}. They were introduced by Howie and Rubinstein~\cite{hr16}. Other knots satisfying \refdef{AltKnots} include generalised alternating links in $S^3$ considered by Ozawa~\cite{oza06}, which are required to be strongly prime, giving restrictions on edge representativity and forcing regions of $F\setminus\pi(L)$ to be disks~\cite{how15t}. 
Knots satisfying properties \eqref{Itm:Alternating}, \eqref{Itm:Connected}, and \eqref{Itm:slope}, include the toroidally alternating knots considered by Adams~\cite{ada94}, required to lie on a Heegaard torus, with disk regions of $F\setminus \pi(L)$. They also include the alternating knots on a Heegaard surface $F$ for $S^3$ considered by Hayashi \cite{hay95}, which are required again to have disk regions of $F\setminus\pi(L)$. Additionally they include the alternating projections of links onto their Turaev surfaces~\cite{dfk08}, which also have disk regions.
By contrast, weakly generalised alternating links and the links of \refdef{AltKnots} are not required to lie on a Heegaard surface, and regions of $F\setminus\pi(L)$ are not required to be disks.

\begin{definition}\label{Def:WeaklyGeneralisedAlternating}
  Let $\pi(L)$ be a reduced alternating diagram on $F$ in $Y$. If further, 
  \begin{enumerate}
  \item[(5)]\label{Itm:CheckCol} $\pi(L)$ is checkerboard colourable, and
  \item[(6)]\label{Itm:Rep} the representativity $r(\pi(L),F)\geq 4$,
  \end{enumerate}
  we say that $\pi(L)$ is a \emph{weakly generalised alternating link diagram} on $F$ in $Y$, and $L$ is a \emph{weakly generalised alternating link}. 
\end{definition}

These conditions are sufficient to show the following.

\begin{theorem}[Howie~\cite{how15t}]\label{Thm:wgaprime}
Let $\pi(L)$ be a weakly generalised alternating projection of a link $L$ in $S^3$. Then $L$ is a nontrivial, nonsplit, prime link. 
\end{theorem}

Theorem~\ref{Thm:wgaprime} generalises a similar theorem of Menasco for prime alternating links~\cite{men84}. For generalised alternating links this was known by Hayashi~\cite{hay95} and Ozawa~\cite{oza06}.

\section{Angled decompositions}\label{Sec:Chunks}
The usual alternating knot and link complements in $S^3$ have a well-known polyhedral decomposition, described by Menasco \cite{men83} (see also \cite{lac00}). The decomposition can be generalised for knots that are reduced alternating on a generalised projection surface $F$, but the pieces are more complicated than polyhedra. This decomposition is similar to a decomposition into angled blocks defined by Futer and Gu{\'e}ritaud \cite{fg09}, but is again more general. A similar decomposition was done for a particular link in \cite{ckp16}, and more generally for links in the manifold $T^2\times I$ in \cite{ckp}. In this section, we describe the decomposition in full generality. 

\subsection{A decomposition of alternating links on $F$}
A \emph{crossing arc} is a simple arc in the link complement running from an overcrossing to its associated undercrossing. 
The polyhedral decomposition of Menasco \cite{men83} can be generalised as follows. 

\begin{proposition}\label{Prop:AltChunkDecomp}
Let $Y$ be a compact, orientable, irreducible 3-manifold containing a generalised projection surface $F$ such that $\bdy Y$ is incompressible in $Y\setminus N(F)$ whenever $\bdy Y\not=\emptyset$.
Let $L$ be a link with a generalised diagram $\pi(L)$ on $F$ satisfying properties \eqref{Itm:Alternating}, \eqref{Itm:Connected}, and \eqref{Itm:slope} of \refdef{AltKnots}. 
Then $Y\setminus L$ can be decomposed into pieces such that:
  \begin{itemize}
  \item Pieces are homeomorphic to components of $Y\setminus N(F)$, where $N(F)$ denotes a regular open neighbourhood of $F$, except each piece has a finite set of points removed from $\bdy (Y\setminus N(F))$ (namely the ideal vertices below). 
  \item On each copy of each component $F_i$ of $F$, there is an embedded graph with vertices, edges, and regions identified with the diagram graph $\pi(L)\cap F_i$. All vertices are ideal and 4-valent.
  \item To obtain $Y\setminus L$, glue pieces as follows. Each region of $F_i\setminus\pi(L)$ is glued to the corresponding region on the opposite copy of $F_i$ by a homeomorphism that is the identity composed with a rotation along the boundary. The rotation takes an edge of the boundary to the nearest edge in the direction of that boundary component's orientation, with orientation as in \refdef{Alternating}. 
  \item Edges correspond to crossing arcs, and are glued in fours. At each ideal vertex, two opposite edges are glued together.
  \end{itemize}
\end{proposition}

\begin{proof}
Consider a crossing arc of the diagram. Sketch its image four times in the diagram at a crossing, as on the left of \reffig{EdgeIdentifications}.
On each side of a component $F_i$, the link runs through overstrands and understrands in an alternating pattern. Pull the crossing arcs flat to lie on $F_i$, with the overstrand running between two pairs of crossing arcs as shown in \reffig{EdgeIdentifications}, left. Note the pattern of overstrands and understrands will look exactly opposite on the two sides of $F_i$. On each side of $F_i$, identify each pair of crossing arcs that are now parallel on $F_i$. These are the edges of the decomposition. Viewed from the opposite side of $F_i$, overcrossings become undercrossings, and exactly the opposite edges are identified. See \reffig{EdgeIdentifications}, right.

\begin{figure}
  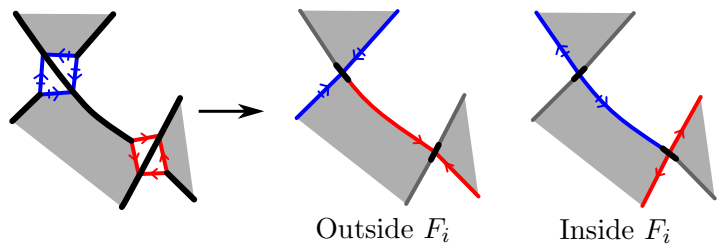
  \caption{Left: Crossing arcs are split into four edges. Middle and right: edges are identified in pairs, with opposite pairs identified on either side of $F_i$.}
  \label{Fig:EdgeIdentifications}
\end{figure}

Now slice along $F_i$. This cuts $Y\setminus L$ into pieces homeomorphic to $Y\setminus N(F)$ (here, we use condition \eqref{Itm:Connected} of \refdef{AltKnots} to conclude that each $F_i$ appears). Each side of $F_i$ is marked with overstrands of $\pi(L)$ and a pair of crossing arcs adjacent to each overcrossing. 

On each side of $F_i$, shrink each overstrand of the diagram to the crossing vertex of $\pi(L)$ corresponding to its overcrossing. Each overstrand thus becomes a single ideal vertex of the decomposition. (And by \eqref{Itm:slope} of \refdef{WeaklyGeneralisedAlternating}, each strand of the link meets a crossing and so is divided up into ideal vertices.)
Note that the edges are pulled to run between vertices following the edges of the diagram graph.

Faces are regions of the diagram graph, and their gluing matches an edge on a region on one side of $F_i$ to an adjacent edge on the same region on the other, in the direction determined by the alternating orientation. 
This gluing is the same as for regular alternating links, and has been likened to a gear rotation \cite{thu79}. Our requirement that regions of a diagram be alternating ensures this gear rotation gluing holds even for faces with multiple boundary components.

All edges coming from a crossing arc are identified together. Thus there are four edges identified to one. At each ideal vertex, a pair of opposite edges are glued together. 
\end{proof}

When our knot or link is checkerboard colourable, we obtain two checkerboard surfaces. 

\begin{definition}\label{Def:CheckerboardSurfaces}
Let $Y$ be a compact, irreducible, orientable 3-manifold with generalised projection surface $F$. Let $\pi(L)$ be a knot or link diagram that is reduced alternating on $F$ and that is also checkerboard colourable. Give $F\setminus\pi(L)$ the checkerboard colouring into white and shaded regions. The resulting coloured regions correspond to white and shaded surfaces with boundary that can be embedded in $(F\times I) \setminus L \subset Y\setminus L$. Complete shaded regions into a spanning surface for $L$ by joining two shaded regions adjacent across a crossing of $\pi(L)$ by a twisted band in $(F\times I)\setminus L$. Similarly for white regions. The two surfaces that arise are the \emph{checkerboard surfaces} of $\pi(L)$, white and shaded. They intersect in crossing arcs. 
\end{definition}

The alternating property and condition \refitm{slope} in \refdef{AltKnots} ensures that the two checkerboard surfaces have distinct slopes on each component of $L$.

\begin{proposition}\label{Prop:AltChunkDecompCheck}
Let $L$ be a link with a generalised diagram $\pi(L)$ on $F$ in a compact, orientable, irreducible 3-manifold $Y$ satisfying properties \eqref{Itm:Alternating}, \eqref{Itm:Connected}, and \eqref{Itm:slope} of \refdef{AltKnots}, and suppose that $\pi(L)$ is checkerboard colourable.
Then the regions of the decomposition of \refprop{AltChunkDecomp} are coloured, shaded and white, and the gluing of faces rotates each boundary component once, in the clockwise direction for white faces, in the counterclockise direction for shaded faces. 
\end{proposition}

\begin{proof}
This follows from the relationship between orientation and alternating boundary components on white and shaded faces.
\end{proof}

\subsection{Defining angled chunks} We now define a decomposition of a 3-manifold for which \refprop{AltChunkDecomp} is an example. 

\begin{definition}\label{Def:Chunk}
  A \emph{chunk} $C$ is a compact, oriented, irreducible 3-manifold with boundary $\bdy C$ containing an embedded (possibly disconnected) non-empty graph $\Gamma$ with all vertices having valence at least 3. We allow components of $\bdy C$ to be disjoint from $\Gamma$.

Regions of $\bdy C\setminus \Gamma$ are called \emph{faces}, despite not necessarily being simply connected.
A face that arises from a component of $\bdy C$ that is disjoint from $\Gamma$ is called an \emph{exterior face}, and we require any exterior face to be incompressible in $C$. Other faces are called \emph{interior faces}.

  A \emph{truncated chunk} is a chunk for which a regular neighbourhood of each vertex of $\Gamma$ has been removed.
This produces new faces, called \emph{boundary faces}, and new edges bordering boundary faces called \emph{boundary edges}. Note boundary faces are homeomorphic to disks.
  
  A \emph{chunk decomposition} of a 3--manifold $M$ is a decomposition of $M$ into chunks, such that $M$ is obtained by gluing chunks by homeomorphisms of (non-boundary, non-exterior) faces, with edges mapping to edges homeomorphically.
\end{definition}

Note we do not require faces to be contractible, we allow bigon faces, and we allow incompressible tori to be embedded in chunks. These are all more general than Futer and Gu{\'e}ritaud's blocks \cite{fg09}.

\refprop{AltChunkDecomp} implies that for a reduced alternating knot or link $L$ on $F$ in $Y$, the complement $Y\setminus L$ has a chunk decomposition.
Note that for any chunk in the decomposition, the graph $\Gamma$ is the diagram graph of the link, and so each vertex has valence four. Thus for any truncated chunk arising from a reduced alternating link $L$ on $F$ in $Y$, all boundary faces are squares. We will use this fact frequently in our applications below. 

\begin{remark}
We add a word on notation. The manifold $Y\setminus L$ is not a compact manifold; the chunks in its decomposition glue to form $Y\setminus L$ with ideal vertices on $L$. We also frequently need to consider the compact manifold $Y\setminus N(L)$, where $N(\cdot)$ denotes a regular open neighbourhood. We will denote this manifold by $X(L):=Y\setminus N(L)$, or more simply by $X$. It is the \emph{exterior} of the link $L$ in $Y$. Truncated chunks glue to form this compact manifold.
\end{remark}

In the case that $\pi(L)$ is checkerboard colourable, we will also be interested in the manifold obtained by leaving the shaded faces of the chunk deomposition of $Y\setminus L$ unglued. This is homeomorphic to $(Y\setminus L)\setminus N(\Sigma)$, where $\Sigma$ is the shaded checkerboard surface. More accurately, we will leave shaded faces of \emph{truncated} chunks unglued, and so the result is homeomorphic to $X\setminus N(\Sigma)$. We denote this compact manifold by $X\cut\Sigma$. Its boundary consists of $\bdy N(\Sigma) =\widetilde{\Sigma}$, the double cover of $\Sigma$, and remnants of $\bdy N(L)$ coming from boundary faces, as well as exterior faces coming from $\bdy Y$. 
The portion of the boundary $\bdy N(L) \setminus N(\Sigma)$ is called the \emph{parabolic locus}.

\begin{definition}\label{Def:BoundedChunk}
Let $M$ be a compact orientable 3-manifold containing a properly embedded surface $\Sigma$. 
A \emph{bounded chunk decomposition} of the manifold $M\cut\Sigma := M\setminus N(\Sigma)$ is a decomposition of $M\cut\Sigma$ into truncated chunks, with
boundary faces, exterior faces, and interior faces as before, only now faces lying on $\widetilde{\Sigma}\subset \bdy(M\cut\Sigma)$ are left unglued.

The faces that are glued are still called \emph{interior faces}. Those left unglued, lying on $\widetilde{\Sigma}$, are called \emph{surface faces}.

Edges are defined to be \emph{boundary edges} if they lie between boundary faces and other faces, \emph{interior edges} if they lie between two interior faces, and \emph{surface edges} if they lie between a surface face and an interior face. Surface faces are not allowed to be adjacent along an edge.

Finally, the restriction of the gluing to surface faces identifies each surface edge to exactly one other surface edge, and produces $\widetilde{\Sigma}$.
\end{definition}

\subsection{Normal surfaces}
Although chunks are more general than blocks, we can define normal surfaces inside them. The following definition is modified slightly from Futer--Gu{\'e}ritaud.

\begin{definition}\label{Def:NormalSurface}
For $C$ a truncated chunk, and $(S,\bdy S)\subset (C, \bdy C)$ a properly embedded surface, we say $S$ is \emph{normal} if it satisfies:
  \begin{enumerate}
  \item\label{Itm:IncomprNorm} Each closed component of $S$ is incompressible in $C$.
  \item $S$ and $\bdy S$ are transverse to all faces, boundary faces, and edges of $C$.
  \item\label{Itm:BdryNoDisk} If a component of $\bdy S$ lies entirely in a face (or boundary face) of $C$, then it does not bound a disk in that face (or boundary face).
  \item\label{Itm:NoArcEndptsEdge} If an arc $\gamma$ of $\bdy S$ in a face (or boundary face) of $C$ has both endpoints on the same edge, then the arc $\gamma$ along with an arc of the edge cannot bound a disk in that face (or boundary face). 
  \item\label{Itm:NoArcEndptsAdj} If an arc $\gamma$ of $\bdy S$ in a face of $C$ has one endpoint on a boundary edge and the other on an adjacent interior or surface edge, then $\gamma$ cannot cut off a disk in that face. 
  \end{enumerate}
  Given a chunk decomposition of $M$, a surface $(S,\bdy S)\subset (M,\bdy M)$ is called \emph{normal} if for every chunk $C$, the intersection $S\cap C$ is a (possibly disconnected) normal surface in $C$.
\end{definition}

We have made modifications to Futer--Gu{\'e}ritaud's definition in items \eqref{Itm:BdryNoDisk}, \eqref{Itm:NoArcEndptsEdge}, and \eqref{Itm:NoArcEndptsAdj}: for item \eqref{Itm:BdryNoDisk}, we may in fact have components of $\bdy S$ that lie in a single face; however in such a case, the face must not be contractible and $\bdy S$ must be a nontrivial curve in that face. Similarly for \eqref{Itm:NoArcEndptsEdge}, and \eqref{Itm:NoArcEndptsAdj}.
Note also that since boundary faces are already disks, items \eqref{Itm:NoArcEndptsEdge} and \eqref{Itm:NoArcEndptsAdj} agree with Futer--Gu{\'e}ritaud's definition for boundary faces. 

Also, note that we consider a 2-sphere to be incompressible if it does not bound a ball, hence the fact that a chunk is irreducible by definition, along with item \eqref{Itm:IncomprNorm} of \refdef{NormalSurface} implies that the intersection of a normal surface with a chunk has no spherical components.

\begin{theorem}\label{Thm:NormalForm}
  Let $M$ be a manifold with a chunk or bounded chunk decomposition. 
  \begin{enumerate}
  \item If $M$ is reducible, then $M$ contains a normal 2-sphere.
  \item If $M$ is irreducible and boundary reducible, then $M$ contains a normal disk.
  \item If $M$ is irreducible and boundary irreducible, then any essential surface in $M$ can be isotoped into normal form.
  \end{enumerate}
\end{theorem}

\begin{proof}
  The proof is a standard innermost disk, outermost arc argument. In the case of a chunk decomposition, it follows nearly word for word the proof of \cite[Theorem~2.8]{fg09}; we leave the details to the reader.

  In the case of a bounded chunk decomposition, an essential surface $S$ can no longer be isotoped through surface faces. We modify the proof where required to avoid such moves. First, if a component of $\bdy S$ lies entirely in a surface face and bounds a disk in that face, consider an innermost such curve. Since $S$ is incompressible, that curve bounds a disk in $S$ as well, hence $S$ has a disk component, parallel into a surface face, contradicting the fact that it is essential.

  If an arc of intersection of $S$ with a face has both its endpoints on the same surface edge, and the arc lies in an interior face, then an outermost such arc and the edge bound a disk $D$ with one arc of $\bdy D$ on $S$ and one arc on a surface face. Because $S$ is essential, it is boundary incompressible; it follows that the arc of intersection can be pushed off. A similar argument implies that an arc of intersection of $S$ with an interior face that has one endpoint on a boundary edge and one on a surface edge can be pushed off. For all other arcs of intersection with endpoints on one edge, or an edge and adjacent boundary edges, the argument follows just as before. 
\end{proof}

\subsection{Angled chunks and combinatorial area}

\begin{definition}\label{Def:AngledChunk}
  An \emph{angled chunk} is a truncated chunk $C$ such that each edge $e$ of $C$ has an associated interior angle $\alpha(e)$ and exterior angle $\epsilon(e) = \pi-\alpha(e)$ satisfying:
  \begin{enumerate}
  \item\label{Itm:0toPi} $\alpha(e) \in (0,\pi)$ if $e$ is not a boundary edge; $\alpha(e) = \pi/2$ if $e$ is a boundary edge. 
  \item\label{Itm:VertexSum} For each boundary face, let $e_1, \dots, e_n$ denote the non-boundary edges with an endpoint on that boundary face. Then $\sum_{i=1}^n \epsilon(e_i) = 2\pi.$
  \item\label{Itm:NormalDiskSum} For each normal disk in $C$ whose boundary meets edges $e_1, \dots, e_p$,  $\sum_{i=1}^p \epsilon(e_i) \geq 2\pi.$
  \end{enumerate}

  An \emph{angled chunk decomposition} of a 3--manifold $M$ is a subdivision of $M$ into angled chunks, glued along interior faces, such that $\sum \alpha(e_i)=2\pi$, where the sum is over the interior edges $e_i$ that are identified under the gluing.

  A \emph{bounded angled chunk decomposition} satisfies all the above, and in addition $\sum \alpha(e_i)=\pi$ for edges identified to any surface edge under the gluing. That is, if edges $e_1, \dots, e_n$ are identified and all the $e_i$ are interior edges, then $\sum\alpha(e_i)=2\pi$. If two of the edges are surface edges, then $\sum\alpha(e_i)=\pi$. 
\end{definition}

Again our definition is weaker than that of Futer--Gu{\'e}ritaud: in item \refitm{NormalDiskSum}, they require the sum to be strictly greater than $2\pi$ unless the disk is parallel to a boundary face.

\begin{definition}\label{Def:CombinatorialArea}
  Let $C$ be an angled chunk. 
  Let $(S,\bdy S)$ be a normal surface in $(C,\bdy C)$, and let $e_1, \dots, e_n$ be edges of the truncated chunk $C$ met by $\bdy S$, listed with multiplicity. The \emph{combinatorial area} of $S$ is defined to be
  \[ a(S) = \sum_{i=1}^n \epsilon(e_i) - 2\pi\chi(S). \]
  Given a chunk decomposition of $M$ and a normal surface $(S',\bdy S')\subset (M,\bdy M)$, write $S'=\bigcup_{j=1}^m S_j$ where each $S_j$ is a normal surface embedded in a chunk. Define $a(S') = \sum_{j=1}^m a(S_j)$.
\end{definition}

\begin{proposition}\label{Prop:NonnegArea}
  Let $S$ be a connected orientable normal surface in an angled chunk $C$. Then $a(S)\geq 0$. Moreover, if $a(S)=0$, then $S$ is either:
  \begin{enumerate}
  \item[(a)] a disk with $\sum \epsilon(e_i)=2\pi$,
  \item[(b)] an annulus with $\bdy S$ meeting no edges of $\Gamma$, hence the two components of $\bdy S$ lie in non-contractible faces of $C$, or 
  \item[(c)] an incompressible torus disjoint from $\bdy C$. 
  \end{enumerate}
\end{proposition}

\begin{proof}
By definition of combinatorial area, if $\chi(S)<0$, then $a(S)>0$. So we need only check cases in which $\chi(S)\geq 0$.

Note first that $S$ cannot be a sphere because $C$ is required to be irreducible. 
If $S$ is a torus, then it does not meet $\bdy C$ and must be incompressible in $C$. This is item (c).

Suppose now that $S$ is a disk. Then $\bdy S$ meets $\bdy C$,
and $\bdy S$ does not lie completely on an exterior face because such faces are required to be incompressible in $C$. Then because $S$ is in normal form, \refitm{NormalDiskSum} of \refdef{AngledChunk} implies that the sum of exterior angles of edges meeting $\bdy S$ is at least $2\pi$. Thus the combinatorial area of the disk is at least $0$. If the combinatorial area equals zero, then the sum of exterior angles meeting $\bdy S$ must be exactly $2\pi$, as required.

Finally suppose $S$ is an annulus. Then $\chi(S)=0$, so the combinatorial area $a(S)\geq 0$. If it equals zero, then the sum of exterior angles of edges meeting $\bdy S$ is also zero, hence $\bdy S$ meets no edges. Let $\gamma_1$ and $\gamma_2$ denote the two components of $\bdy S$. Then each $\gamma_i$ lies entirely in a single face. Because $\gamma_i$ cannot bound a disk in that face, the face is not contractible. 
\end{proof}

\begin{proposition}[Gauss--Bonnet]\label{Prop:GaussBonnet}
  Let $(S,\bdy S)\subset(M,\bdy M)$ be a surface in normal form with respect to an angled chunk decomposition of $M$. Then
  \[ a(S) = -2\pi\chi(S).\]
  Similarly, let $(S,\bdy S)\subset (M\cut\Sigma,\bdy (M\cut\Sigma))$ 
  be a surface in normal form with respect to a bounded angle chunk decomposition of $M\cut\Sigma$. Let $p$ denote the number of times $\bdy S$ intersects a boundary edge adjacent to a surface face. Then
  \[ a(S) = -2\pi\chi(S) + \frac{\pi}{2}\,p. \]
\end{proposition}

\begin{proof}
The proof is basically that of Futer and Gu{\'e}ritaud \cite[Prop.~2.11]{fg09}, except we need to consider an additional case for angled chunks, and surface faces for bounded angled chunks. We briefly work through their proof, and check that it holds in our more general setting.

As in \cite[Prop.~2.11]{fg09}, consider components of intersection $\{S_1, \dots, S_n\}$ of $S$ with chunks, and let $S'$ be obtained by gluing some of the $S_i$ along some of their edges, so $S'$ is a manifold with polygonal boundary. At a vertex on the boundary of $S'$, one or more of the $S_i$ meet, glued along faces of chunks. Define the interior angle $\alpha(v)$ of $S'$ at such a point $v$ to be the sum of the interior angles of the adjacent $S_i$, and define the exterior angle to be $\epsilon(v) = \pi-\alpha(v)$ (note $\epsilon(v)$ can be negative). We prove, by induction on the number of edges glued, that
\begin{equation}\label{Eqn:AngleSum}
  a(S') = \sum a(S_{i_k}) = \sum_{v\in\bdy S'}\epsilon(v) - 2\pi\chi(S').
\end{equation}

As a base case, if no edges are glued, then \refeqn{AngleSum} follows by \refdef{CombinatorialArea}.

Let $\nu$ be the number of vertices $v$ on $S'$, let $\theta$ be the sum of all interior angles along $\bdy S'$, and let $\chi$ be the Euler characteristic of $S'$, so the right hand side of \refeqn{AngleSum} is $\nu\pi -\theta - 2\pi\chi$. Futer and Gu{\'e}ritaud work through several cases: two edges are glued with distinct vertices; two edges are glued that share a vertex; edges are glued to close a bigon; and a monogon component is glued. None of these moves change \refeqn{AngleSum}.
We have an additional case, namely when $\bdy S_1'$ is glued to $\bdy S_2'$ along simple closed curves, each contained in a single face of a chunk. In this case, $\nu$, $\theta$, and $\chi$ will be unchanged. Thus \refeqn{AngleSum} holds.

When $S$ is a normal surface in a (regular) angled chunk decomposition, let $S'=S$ in \refeqn{AngleSum}. Then all $\epsilon(v)$ come from boundary edges, and equal $\pi-(\pi/2+\pi/2)=0$; this comes from the fact that exterior angles on boundary edges are always $\pi/2$ in \refdef{AngledChunk}.

In the bounded angled chunk case, as above, on boundary edges meeting interior faces we have $\epsilon(v)=\pi-(\pi/2+\pi/2)=0$. On surface edges the sum of interior angles is $\pi$, hence $\epsilon(v)=\pi-\pi=0$. On the $p$ boundary faces meeting surface faces, $\epsilon(v)=\pi-\pi/2=\pi/2$. 
\end{proof}

\begin{theorem}\label{Thm:IrredBdyIrred}
  Let $(M,\bdy M)$ be a compact orientable 3-manifold with an angled chunk decomposition.
  Then $\bdy M$ consists of exterior faces and components obtained by gluing boundary faces; those components coming from boundary faces are homeomorphic to tori. Finally, $M$ is irreducible and boundary irreducible.
\end{theorem}

\begin{proof}
If $M$ is reducible or boundary reducible, it contains an essential 2-sphere or disk. \refthm{NormalForm} implies it contains one in normal form. \refprop{GaussBonnet} implies such a surface has negative combinatorial area. This is impossible by \refprop{NonnegArea}.

Besides exterior faces, each component of $\bdy M$ is tiled by boundary faces. A normal disk parallel to a boundary face has combinatorial area $0$. These disks glue to a surface of combinatorial area $0$ parallel to $\bdy M$. By \refprop{GaussBonnet}, $\chi(\bdy M)=0$. Since $M$ is orientable, each component of $\bdy M$ is a torus.
\end{proof}


\subsection{Chunk decomposition for generalisations of alternating links}

\begin{proposition}\label{Prop:AngledChunkDecomp}
Let $\pi(L)$ be a reduced alternating diagram of a link $L$ on $F$ in $Y$.
Suppose that the representativity satisfies $r(\pi(L),F)\geq 4$. Label each edge of the chunk decomposition of \refprop{AltChunkDecomp} with interior angle $\pi/2$ (and exterior angle $\pi/2$). Then the chunk decomposition is an angled chunk decomposition.

Suppose further that $\pi(L)$ is checkerboard colourable on $F$, so $\pi(L)$ is weakly generalised alternating on $F$. Let $\Sigma$ be one of the checkerboard surfaces associated to $\pi(L)$. Then $X\cut\Sigma$ admits a bounded angled chunk decomposition, with the same chunks as in \refprop{AltChunkDecomp}, but with faces corresponding to $\Sigma$ (white or shaded) left unglued. 
\end{proposition}

\begin{proof}
We check the conditions of the definition of an angled chunk, \refdef{AngledChunk}.
The first two conditions are easy: $\pi/2\in(0,\pi)$, and each ideal vertex of a chunk is 4-valent, so the sum of the exterior angles of the edges meeting a boundary face corresponding to that vertex is $4\cdot \pi/2 = 2\pi$, as required. For the third condition, we need to show that if a curve $\gamma$ bounds a normal disk $D$ in the truncated chunk, meeting edges $e_1, \dots, e_n$, then $\sum_i \epsilon(e_i)\geq2\pi$.

Suppose first that $D$ is not a compressing disk for $F$, so it is parallel into $F$. Then boundary $\gamma$ of $D$ must meet an even number of edges.
If $\gamma$ meets zero edges, then by \refitm{BdryNoDisk} of the definition of normal, \refdef{NormalSurface}, it lies in a face that is not simply connected, and thus $\gamma$ bounds a disk in $F$ that contains edges and crossings of $\pi(L)$, with edges and crossings exerior to the disk as well. Isotope $\gamma$ slightly in this disk so that it crosses exactly one edge of $\pi(L)$ twice. Then we have a disk in $F$ whose boundary meets only two edges of $\pi(L)$ but with crossings contained within (and without) the disk. This contradicts the fact that $\pi(L)$ is weakly prime.

If $\gamma$ meets only two edges, then there are two cases. First, if neither edge is a boundary edge, then $\gamma$ defines a curve on $F$ bounding a disk in $F$ meeting only two edges of $\pi(L)$. Because the diagram is weakly prime, there must be no crossings within that disk, or if $F$ is a 2-sphere, there must be a disk on the opposite side containing no crossings. But then its boundary violates condition \refitm{NoArcEndptsEdge} of the definition of normal. In the second case, one of the edges and hence both edges meeting $\gamma$ are boundary edges. Then $\gamma$ defines a curve on $F$ meeting $\pi(L)$ exactly once in a single crossing. Push the disk $D$ slightly off this crossing, so $\bdy D$ meets $\pi(L)$ in exactly two edges with a single crossing lying inside $D$. If $F$ is not a 2-sphere, this immediately contradicts the fact that the diagram is weakly prime. If $F$ is a 2-sphere, because $\pi(L)$ must have more than one crossing, again this contradicts the fact that the diagram is weakly prime. 
Thus $\gamma$ must meet at least four edges, and $\sum_i\epsilon(e_i) \geq 4\cdot \pi/2 = 2\pi$.

Now suppose that $D$ is a compressing disk for $F$. Then again $\gamma$ determines a curve on $F$ meeting $\pi(L)$, bounding a compressing disk for $F$. If $\gamma$ meets no boundary faces, the fact that $r(\pi(L),F)\geq 4$ implies that $\gamma$ meets at least four edges of the chunk, so $\sum_i\epsilon(e_i)\geq 4\cdot \pi/2 = 2\pi$.  If $\gamma$ meets a boundary face, then it meets two boundary edges on that boundary face. Isotope $\bdy D$ through the boundary face and slightly outside; let $\beta$ be the result after isotopy. Note the isotopy replaces the two intersections of $\gamma$ with boundary edges by one or two intersections of $\beta$ with edges whose endpoints lie on the boundary face, so $\beta$ meets at most as many edges as $\gamma$. But then $\beta$ defines a curve on $F$ meeting $\pi(L)$, bounding a compressing disk for $F$. Again $r(\pi(L),F)\geq 4$ implies $\beta$ meets at least four edges. It follows that $\gamma$ meets at least four edges. So again $\sum\epsilon(e_i)\geq 2\pi$. 

Finally, for a chunk decomposition of $Y\setminus L$, because edges are glued in fours, the sum of all interior angles glued to an edge class is $4\cdot \pi/2 = 2\pi$, as required.

For the checkerboard colourable case, white or shaded faces left unglued, each non-boundary edge is a surface edge. The sum of interior angles at each such edge is $\pi/2+\pi/2=\pi$. 
\end{proof}

\begin{corollary}\label{Cor:IrredBdryIrred}
If $\pi(L)$ is a reduced alternating diagram of a link $L$ on $F$ in $Y$, and $r(\pi(L),F)\geq 4$, then $Y\setminus L$ is irreducible and boundary irreducible.\qed
\end{corollary}

\refcor{IrredBdryIrred} shows that weakly generalised alternating links in $S^3$ are nontrivial and nonsplit, giving a different proof of this fact than in \cite{how15t}.

\begin{corollary}\label{Cor:NoNormalBigons}
If $\pi(L)$ is a reduced alternating diagram of a link $L$ on $F$ in $Y$ with $r(\pi(L),F)\geq 4$, then the chunks in the decomposition of $Y\setminus L$ contain no normal bigons, i.e.\ no normal disks meeting exactly two interior edges. 
\end{corollary}

\begin{proof}
A normal disk meeting exactly two interior edges would have $\sum \epsilon(e_i) = \pi/2+\pi/2=\pi < 2\pi$, contradicting the definition of an angled chunk decomposition. 
\end{proof}

\begin{definition}\label{Def:Pi1Essential}
A properly embedded surface $S$ in a 3-manifold $M$ is \emph{$\pi_1$-essential} if
  \begin{enumerate}
  \item $\pi_1(S)\to\pi_1(M)$ is injective,
  \item $\pi_1(S,\bdy S)\to\pi_1(M,\bdy M)$ is injective, and
  \item $S$ is not parallel into $\bdy M$.
  \end{enumerate}
\end{definition}

When $Y=S^3$, Howie and Rubinstein proved that the checkerboard surfaces of a weakly generalised alternating link in $S^3$ are essential \cite{hr16}. Using the machinery of angled chunk decompositions, we can extend this result to weakly generalised alternating links in any compact, orientable, irreducible 3-manifold $Y$. 

\begin{theorem}\label{Thm:hress}
Let $\pi(L)$ be a weakly generalised alternating diagram of a link $L$ on a generalised projection surface $F$ in $Y$.
Then both checkerboard surfaces associated to $\pi(L)$ are $\pi_1$-essential in $X(L)$.
\end{theorem}

Ozawa~\cite{oza06} proved a similar theorem for generalised alternating projections in $S^3$. Similarly, Ozawa~\cite{oza11} and Ozawa and Rubinstein~\cite{or12} showed that related classes of links admit essential surfaces, some of which can be viewed as checkerboard surfaces for a projection of the link onto a generalised projection surface (e.g.\ the Turaev surface; see also \cite{fkp13}). The original proof that checkerboard surfaces are essential, for reduced prime alternating planar projections in $S^3$, was due to Aumann in 1956~\cite{aum56}.

\begin{proof}[Proof of \refthm{hress}]
Let $\Sigma$ be a checkerboard surface. Recall $X$ denotes the link exterior, $X=Y\setminus N(L)$. If $\pi_1(\Sigma)\to\pi_1(X)$ is not injective, then $\pi_1(\widetilde{\Sigma})\to\pi_1(X)$ is not injective. 
Because $\widetilde{\Sigma} = \bdy N(\Sigma)$ is 2-sided, by the loop theorem there exists a properly embedded essential disk $D$ in $X\cut\Sigma$ with boundary on $\widetilde{\Sigma}$. We may put $D$ into normal form with respect to the bounded chunk decomposition of $X\cut \Sigma$. Note $\bdy D$ does not meet the parabolic locus $P$ of $X\cut \Sigma$, which consists of the boundary components $\bdy N(L)\cap\bdy(X\cut \Sigma)$. Thus \refprop{GaussBonnet} implies that $a(D) = -2\pi$. On the other hand, any normal component of $D$ in a chunk has combinatorial area at least $0$, by definition of combinatorial area, \refdef{CombinatorialArea}, and definition of an angled chunk, \refdef{AngledChunk}. This is a contradiction.

Now suppose $\pi_1(\Sigma,\bdy\Sigma)\to\pi_1(X, \bdy N(L))$ is not injective. Then again the loop theorem implies there exists an embedded essential disk $E$ in $X\cut \Sigma$ with $\bdy E$ consisting of an arc on $\widetilde{\Sigma}$ and an arc on $P$. Put $E$ into normal form with respect to the chunk decomposition. \refprop{GaussBonnet} implies $a(E) = -2\pi\chi(E) + \pi = -\pi$. Again this is a contradiction.
\end{proof}

\section{Detecting hyperbolicity}\label{Sec:Hyperbolic}
Our next application of the chunk decomposition is to determine conditions that guarantee that a reduced alternating link on a generalised projection surface is hyperbolic. Thurston~\cite{thu82} proved that a 3-manifold
has hyperbolic interior whenever it is irreducible, boundary irreducible, atoroidal, and anannular. Using this result, Futer and Gu{\'e}ritaud show that a 3-manifold with an angled block decomposition is hyperbolic~\cite{fg09}. But when we allow a more general angled chunk decomposition, the manifold may contain essential tori and annuli. The main result of this section, \refthm{Hyperbolic}, restricts these.

\begin{definition}\label{Def:BdyAnannular}
The manifold $Y\setminus N(F)$ is \emph{$\bdy$-annular} if it contains a properly embedded essential annulus with both boundary components in $\bdy Y$. Otherwise, it is \emph{$\bdy$-anannular}.
\end{definition}

\begin{theorem}\label{Thm:Hyperbolic}
Let $\pi(L)$ be a weakly generalised alternating diagram of a link $L$ on a generalised projection surface $F$ in a 3-manifold $Y$.
Suppose $F$ has genus at least one, all regions of $F\setminus\pi(L)$ are disks, and $\hat{r}(\pi(L),F)>4$. Then:
\begin{enumerate}
\item $Y\setminus N(L)$ is toroidal if and only if $Y\setminus N(F)$ is toroidal.
\item $Y\setminus N(L)$ is annular if and only if $Y\setminus N(F)$ is $\bdy$-annular.
\item Otherwise, the interior of $Y\setminus N(L)$ admits a hyperbolic structure.
\end{enumerate}
\end{theorem}

A few remarks on the theorem. 
Hayashi~\cite{hay95} showed that if $F$ is a connected generalised projection surface in a closed $3$-manifold $Y$, the regions of $F\setminus\pi(L)$ are disks, $\pi(L)$ is weakly prime, and in addition, if $e(\pi(L),F)>4$, then $Y\setminus L$ contains no essential tori. 
\refthm{Hyperbolic} generalises Hayashi's theorem even in the case that $F$ is a Heegaard surface for $S^3$, for if $\hat{r}(\pi(L),F)>4$, the edge-representativity could still be $2$ or $4$.

Also, we cannot expect to do better than \refthm{Hyperbolic}, for Hayashi~\cite{hay95} gives an example of an 8-component link in $S^3$ that is reduced alternating on a Heegaard torus $F$, with all regions of $F\setminus\pi(L)$ disks and $r(\pi(L),F)=\hat{r}(\pi(L),F)=4$, such that the exterior of $L$ admits an embedded essential torus. However, in the results below, we do give restrictions on the forms of reduced alternating links on $F$ that are annular or toroidal. 

\begin{example}\label{Exa:T2xI}
Consider the case $Y=T^2\times[-1,1]$, a thickened torus, with generalised projection surface $F=T^2\times\{0\}$. Because $F$ has no compressing disks in $Y$, the representativity of any alternating link on $F$ is infinite. Thus provided an alternating diagram on $F$ has at least one crossing, is weakly prime, and is checkerboard colourable, it is weakly generalised alternating. If all regions of the diagram are disks, then \refthm{Hyperbolic} immediately implies that the link is hyperbolic. 

Recently, others have studied the hyperbolic geometry of restricted classes of alternating links on $T^2\times\{0\}$ in $T^2\times[-1,1]$, for example coming from uniform tilings~\cite{acm17, ckp}, and diagrams without bigons~\cite{ckp}. \refthm{Hyperbolic} immediately proves the existence of a hyperbolic structure of such links. However, the results in \cite{acm17, ckp} give more details on their geometry.

Similarly, \refthm{Hyperbolic} implies that a weakly prime, alternating, checkerboard colourable diagram on a surface $S\times\{0\}$ in $S\times[-1,1]$ is hyperbolic, for more general $S$. The hyperbolicity of such links has also been considered by Adams \emph{et al}~\cite{Adams:ThickenedSfces}. 
\end{example}


\subsection{Essential annuli}

We first consider essential annuli. Suppose $A$ is an essential annulus embedded in the complement of a link that is reduced alternating on a generalised projection surface $F$, as in \refdef{AltKnots}. Then \refthm{NormalForm} and \refprop{GaussBonnet} imply that $A$ can be put into normal form, with $a(A)=0$. Then \refprop{NonnegArea} implies it is decomposed into disks with area $0$. If $A$ meets $\bdy N(L)$, it must meet boundary faces and edges. Recall that all boundary faces are squares. 

\begin{lemma}\label{Lem:NormalSquare}
Let $\pi(L)$ be a reduced alternating diagram of a link $L$ on a generalised projection surface $F$ in $Y$, as in \refdef{AltKnots}. Suppose further that the representativity satisfies $r(\pi(L),F)\geq 4$.
In the angled chunk decomposition of $X(L)$, a normal disk with combinatorial area zero that meets a boundary face has one of three forms:
\begin{enumerate}
\item either it meets a single boundary face and two non-boundary edges, and runs through opposite boundary edges of the boundary face, or
\item it meets two boundary faces, intersecting two adjacent boundary edges in each, and encircles a single non-boundary edge of the chunk, or
\item it meets two boundary faces in opposite boundary edges of each boundary face. 
\end{enumerate}
If the link is checkerboard colourable, in the second form it runs through faces of opposite colour, and in the third it runs through faces of the same colour. 
\end{lemma}

\begin{proof}
Because all edges have exterior angle $\pi/2$, a disk of combinatorial area zero meets exactly four edges. Hence if it meets a boundary face, it either meets four boundary edges or two boundary edges and two interior edges. 
 
There are various cases to consider. If it meets two boundary edges and two non-boundary edges, then it might meet adjacent boundary edges on the boundary face. In this case, slide the disk slightly off the boundary face so that its two intersections with boundary faces are replaced by an intersection with a single non-boundary edge. This gives a disk meeting three edges of the diagram. The representativity condition ensures the disk is parallel into the surface $F$. But the diagram consists of immersed closed curves on the surface, so they cannot meet the boundary of a disk an odd number of times. Thus in the first case, the disk meets opposite boundary edges of a boundary face.

Now suppose that the disk meets exactly two boundary faces. If it meets both faces in opposite boundary edges, we are in the third case. If it meets one face in opposite boundary edges and one in adjacent boundary edges, again push the disk slightly off the boundary faces to obtain a curve bounding a disk that meets three edges of the diagram. As before, this gives a contradiction. So the only other possibility is that the disk meets two boundary faces and runs through adjacent boundary edges in both. This is the second case. 

In the second case, we may slide the boundary of the normal disk slightly off the boundary face in a way that minimises intersections with non-boundary edges. Thus it will meet non-boundary edges exactly twice; this gives a curve $\gamma$ on $F$ meeting the link diagram exactly twice. The representativity condition implies that $\gamma$ bounds a disk parallel into $F$. Then the fact that the diagram is weakly prime implies that it has no crossings in the interior of the disk, so the curve encircles a single edge as claimed. 
\end{proof}

As in the previous lemma, we will prove many results by considering the boundaries of normal disks on the chunk decomposition. The combinatorics of the gluing of faces in \refprop{AltChunkDecomp} and the fact that the chunk decomposition comes from an alternating link will allow us to obtain restrictions on the diagram of our original link.

Next we determine the form of a normal annulus when one of its normal disks has the first form of \reflem{NormalSquare} and is parallel into $F$.

\begin{lemma}\label{Lem:Annulus}
Let $\pi(L)$ be a reduced alternating diagram of a link $L$ on a generalised projection surface $F$ in a 3-manifold $Y$, and suppose $r(\pi(L),F)\geq 4$. 
Suppose $A$ is a normal annulus in the angled chunk decomposition of the link exterior $X:=Y\setminus N(L)$. 
Suppose $A$ made up of normal squares such that one square $A_i\subset A$ has boundary meeting exactly one boundary face and two non-boundary edges, and $A_i$ is parallel into the surface $F$. Then
  \begin{itemize}
  \item two components of $L$ form a 2-component link with a checkerboard colourable diagram which consists of a string of bigons arranged end to end on some component $F_j$ of $F$, 
  \item a checkerboard surface $\Sigma$ associated to these components of $\pi(L)$ is an annulus, and
  \item a sub-annulus of $A$ has one boundary component running through the core of $\Sigma$, and the other parallel to $\bdy \Sigma$ on $\bdy N(L)$.
  \end{itemize}
\end{lemma}

\begin{proof}
Let $A_i$ be as in the statement of the lemma, so $\bdy A_i$ bounds a disk on the surface of a chunk and meets one boundary face and two non-boundary edges. By \reflem{NormalSquare}, it must intersect opposite boundary edges of the boundary face that it meets.

Now $A_i$ is glued to squares $A_{i+1}$ and $A_{i-1}$ along its sides adjacent to the boundary face. Say $A_i$ is glued to $A_{i-1}$ in a face that we denote $W_{i-1}$, and $A_i$ is glued to $A_{i+1}$ in a face denoted $W_{i+1}$. The boundary $\bdy A_i$ runs through one more non-boundary face, which we denote by $B_i$.

By \refprop{AltChunkDecomp}, the gluing of $A_i$ to $A_{i-1}$ and $A_{i+1}$ will be by rotation in the faces $W_{i-1}$ and $W_{i+1}$, respectively. Since $\bdy A_i$ runs through opposite edges of a boundary face, the faces $W_{i-1}$ and $W_{i+1}$ are opposite across a crossing of $\pi(L)$. Thus the gluing must be by rotation in the same direction (i.e.\ following the same orientation) in these two faces. 

Superimpose $\bdy A_{i-1}$ and $\bdy A_{i+1}$ onto the same chunk as $A_i$. Arcs are as shown in \reffig{GluingSquares}, left. In that figure, faces $W_{i-1}$ and $W_{i+1}$ are coloured white, and faces adjacent to these are shaded. The colouring is for ease of visualisation only; the argument applies even if the link is not checkerboard colourable. 

\begin{figure}
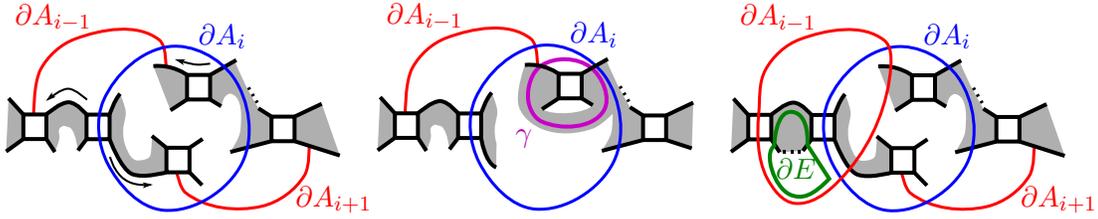
  \caption{On the left, arcs of $\bdy A_{i+1}$ and $\bdy A_{i-1}$ are as shown. Middle, if $\bdy A_i$ meets $\bdy A_{i-1}$ in the face $B_i$ (shaded). On the right, $\bdy A_{i-1}$ must close up as shown.}
  \label{Fig:GluingSquares}
\end{figure}

Two additional comments on the figure: Because $A_i$ is normal, $\bdy A_i$ cannot run from a boundary face to an adjacent edge. This forces arcs of $\bdy A_{i-1}$ and $\bdy A_{i+1}$ to intersect $\bdy A_i$ in faces $W_{i-1}$ and $W_{i+1}$, respectively, as shown. Second, the gluing map is a homeomorphism of these faces. Since $\bdy A_i$ bounds a disk in $F$, the parts of the white face bounded by the arcs of $A_{i-1}$ and $A_{i+1}$ are simply connected.

Note $\bdy A_{i-1}$ has points inside and outside the disk bounded by $\bdy A_i$, so it must cross $\bdy A_i$ at least twice. One crossing of $\bdy A_{i-1}$ is known: it lies in $W_{i-1}$. We claim that $\bdy A_{i-1}$ cannot cross $\bdy A_i$ in the face $B_i$. For suppose by way of contradiction that $\bdy A_{i-1}$ does intersect $\bdy A_i$ in $B_i$. Then $\bdy A_{i-1}$ must run from $W_{i-1}$ across an interior edge into the face $B_i$, and similarly for $\bdy A_i$. Thus the face $B_i$ lies on opposite sides of a crossing. Draw an arc in $B_i$ from one side of the crossing to the other, and connect to an arc in $W_{i-1}$ to form a closed curve $\gamma$; see \reffig{GluingSquares}, middle. 
Note since $\gamma$ is a simple closed curve lying in the disk in $F$ with boundary $\bdy A_i$, $\gamma$ bounds a disk. But $\gamma$ gives a curve meeting the diagram exactly twice with a single crossing inside the disk; because there are also crossings outside the disk, this contradicts the fact that $\pi(L)$ is weakly prime.

Since $A_{i-1}$ and $A_{i+1}$ lie in the same chunk of the decomposition, either they coincide or they are disjoint. If they coincide, and one of $W_{i-1}$ or $W_{i+1}$ is simply-connected, then $A_{i-1}$ contradicts the definition of normal, since an arc of $\bdy A_{i-1}$ runs from an interior edge to an adjacent boundary edge. If neither region is simply-connected, then two components of $\pi(L)$ form a chain of two bigons on some component of $F$, as required. So now assume $A_{i-1}$ and $A_{i+1}$ are distinct and disjoint, so that the images of $\bdy A_{i-1}$ and $\bdy A_{i+1}$ superimposed on the chunk are also disjoint.

Since $\bdy A_{i-1}$ intersects $\bdy A_i$ in the face $W_{i+1}$ but $\bdy A_{i-1}$ is disjoint from $\bdy A_{i+1}$ in this face, it follows that all intersections of $\bdy A_{i-1}$ with $\bdy A_i$ must occur within the disk in $W_{i+1}$ bounded by $\bdy A_i$ outside the arc $\bdy A_{i+1}$. Because this is a disk, we may isotope $\bdy A_{i-1}$ in the disk to meet $\bdy A_i$ exactly once.
Thus $\bdy A_{i-1}$ must be as shown in \reffig{GluingSquares}, right. 

Now consider the dotted line lying within $\bdy A_{i-1}$ in \reffig{GluingSquares}, right. Because $\bdy A_i$ bounds a disk on $F$, and because the portion of the face $W_{i-1}$ bounded by $\bdy A_{i-1}$ is also a disk, we may obtain a new disk $E$ by isotoping $A_{i-1}$ through these two disks and slightly past the interior edges and boundary faces. The boundary of the disk $E$ is shown in \reffig{GluingSquares}, right. Because $r(\pi(L),F)\geq 4$, $E$ must be parallel into $F$. It follows that $A_{i-1}$ is also a disk parallel into $F$. Finally, because $\pi(L)$ is weakly prime, $E\cap\pi(L)$ must be a single arc with no crossings. Let $D_{i-1}$ and $D_i$ be the subdisks of $F$ bounded by $\bdy A_{i-1}$ and $\bdy A_i$ respectively, which are parallel in $C$ to $A_{i-1}$ and $A_i$ respectively. It follows that $D_{i-1}\setminus D_i$ bounds a single bigon region.

Repeat the above argument with $A_{i-1}$ replacing $A_i$, and $A_{i-2}$ replacing $A_{i-1}$. It follows that $\bdy A_{i-2}$ bounds a subdisk $D_{i-2}$ of $F$ parallel to $A_{i-2}$, with $D_{i-2}\setminus D_{i-1}$ bounding a single bigon region. Note also that $\bdy A_{i-2}$ must run through the bigon face bounded by $D_{i-1}\setminus D_i$, as it is disjoint from $A_i$. Repeat again, and so on. It follows that the diagram $\pi(L)$ contains a string of bigons arranged end to end, and the squares $A_j$ each encircle one bigon. The bigons can be shaded, forming a checkerboard colouring. An arc of each $\bdy A_j$ lies on a shaded bigon.

Now boundary faces are arranged in a circle, with bigons between them. Note in the direction of the circle, the boundary faces alternate meeting disks $\{A_j\}$ in one chunk then the other (shown red and blue in \reffig{GluingSquares}). If there are an odd number of bigons, then the disks $\{A_j\}$ overlap in a chunk, contradicting the fact that $A$ is embedded. So there are an even number of bigons. Then this portion of the diagram is a two component link, and the shaded surface $\Sigma$ is an annulus between link components, with arcs of $\bdy A_j$ in the shaded faces gluing to form the core of $\Sigma$. Finally, note the component of $\bdy A$ on the boundary faces never meets the shaded annulus. Hence it runs parallel to the boundary of the annulus $\Sigma$ on $\bdy N(L)$. 
\end{proof}

\begin{theorem}\label{Thm:AnAnnularLink}
Let $\pi(L)$ be a reduced alternating diagram of a link $L$ on a generalised projection surface $F$ in a 3-manifold $Y$. Suppose that $r(\pi(L),F)\geq 4$ and $\hat{r}(\pi(L),F)>4$. If the link exterior $X(L)$ contains an essential annulus $A$ with at least one boundary component on $\bdy N(L)$, then $\pi(L)$ contains a string of bigons on $F$. If $S$ denotes the annulus or M\"obius band formed between the string of bigons, then a component of $\bdy A$ on $\bdy N(L)$ is parallel to the boundary slope of a component of $\bdy S$ on $\bdy N(L)$. 
\end{theorem}

\begin{proof}
Put $A$ into normal form with respect to the angled chunk decomposition of $X(L)$. \refprop{GaussBonnet} implies that $A$ meets chunks in components with combinatorial area zero. Because $\bdy A$ meets $\bdy N(L)$, one such component meets a boundary face. 
Then \refprop{NonnegArea} implies that the chunk decomposition must divide $A$ into disks. It follows that the other boundary component of $A$ cannot lie on an exterior face and so it must also lie on $\bdy N(L)$. 
Let $A_i$ be a normal disk in $A$. \reflem{NormalSquare} implies it has one of three forms.

If $A_i$ has the first form, then it is glued to another disk $A_{i+1}$ of the first form. Because $\hat{r}(\pi(L),F)>4$, one of $A_i$ or $A_{i+1}$ is not a compression disk for $F$. Thus it is parallel into $F$. Then by \reflem{Annulus}, the diagram $\pi(L)$ is as claimed. 
However, no such annulus $A$ actually exists. In the proof of \reflem{Annulus}, we only glued up the normal squares along some of their edges. One edge of $A_i$ lies in a bigon face $B$. Glue this edge to some normal square $A'_i$. Since $A'_i$ is disjoint from $A_{i-1}$ and $A_{i+1}$, it follows that $\bdy A'_i$ lies inside the subdisk of $F$ bounded by either $\bdy A_{i-1}$ or $\bdy A_{i+1}$. But then the only possibility which allows $A'_i$ to be a normal square is if $\bdy A'_i$ meets only interior edges. Hence it is impossible to glue up normal squares to form a properly embedded annulus when one such square has the first form, since we will never arrive at the other boundary component of $A$.

So suppose $\bdy A_i$ meets two boundary faces.
If all normal disks of $A$ are of the second form of \reflem{NormalSquare}, they encircle a single crossing arc, and $A$ is not essential.

If all normal disks of $A$ are of the third form, then $\hat{r}(\pi(L),F)>4$ implies one, say $A_i$, is not a compressing disk so is parallel into $F$. Superimpose three adjacent squares $\bdy A_i$, $\bdy A_{i+1}$, and $\bdy A_{i-1}$ on the boundary of the same chunk. There are two cases: $\bdy A_i$ and $\bdy A_{i-1}$ may be disjoint, or they may intersect.

If all normal disks are of the third form and $\bdy A_i$ and $\bdy A_{i-1}$ are disjoint, then the fact that the diagram is weakly prime implies $\bdy A_i$ must bound a bigon face (as in the proof of \reflem{Annulus}). But then similar arguments, using weakly prime and $r(\pi(L),F)\geq 4$, show $\bdy A_{i-1}$ and $\bdy A_{i+1}$ must also bound bigons. Repeating for these disks, it follows that all $\bdy A_j$ bound bigons, and $\pi(L)$ is a string of bigons. Since $\bdy A_i$ avoids the bigon faces, the boundary components of $\bdy A$ run parallel to the surface made up of the bigons. If there are an odd number of bigons, each bigon is encircled by $\bdy A_j$ for normal disks on both sides of $F$, and the boundary of $A$ runs parallel to the boundary of a regular neighbourhood of the M\"obius band made up of these bigons. If there are an even number, then the bigons form an annulus, with this portion of $\pi(L)$ forming a two component link, and $\bdy A$ running along the link parallel to the annulus. Either case gives the desired result.

Suppose that all normal disks are of the third form and $\bdy A_i$ and $\bdy A_{i-1}$ intersect. Since either $\bdy A_{i-1}$ and $\bdy A_{i+1}$ are disjoint or $A_{i-1}$ and $A_{i+1}$ coincide, the weakly prime and representativity conditions show that $\bdy A_i$ bounds two bigon faces. If $A_{i-1}$ and $A_{i+1}$ coincide, then $\pi(L)$ contains a string of exactly two bigons. If not, using similar ideas to the proof of \reflem{Annulus}, it follows that $\bdy A_{i-1}$ and $\bdy A_{i+1}$ also bound two bigon faces, and so does each $\bdy A_j$. Since $\bdy A_{j-1}$ and $\bdy A_{j+1}$ do not intersect for any $j$, there must be an even number of bigons, and $\pi(L)$ forms a two component link, with $\bdy A$ running along the link parallel to the annulus formed by the bigons, as required.

So suppose there are normal disks of $A$ of both the second and third forms. There must be a disk $S_0$ of the third form glued to one $S_1$ of the second form. Superimpose the boundaries of $S_0$ and $S_1$ on $F$. The rotation of the gluing map on chunks implies that the boundary faces met by $\bdy S_0$ are adjacent to a single edge. Since $r(\pi(L),F)\geq 4$ and the diagram is weakly prime, $\bdy S_0$ must bound a single bigon of $\pi(L)$ as in \reffig{Form23}. (The full argument is nearly identical to the argument that $\bdy A_{i-1}$ bounds a bigon in the proof of \reflem{Annulus}.)

\begin{figure}
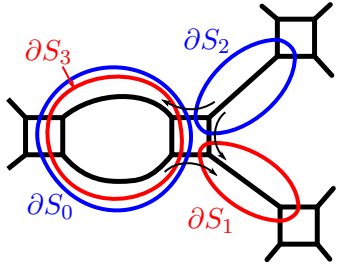
  \caption{Shows how disks of the second and third form must lie in $F$. (Thin black arrows show the direction of the rotation for the gluing map from the blue side to the red.)}
  \label{Fig:Form23}
\end{figure}

Now consider $S_2$ glued to $S_1$ along its edge in the other face met by $S_1$. The disk $S_2$ is disjoint from $S_0$, so $\bdy S_2$ cannot run through opposite sides of the boundary face it shares with $S_0$. Thus $S_2$ must be of the second form, encircling a single interior edge just as $S_1$ does. Finally, $S_2$ glues to some $S_3$ along an edge in its other face. We claim $S_3$ cannot be of the second form, encircling an interior edge, since if it did, it would glue to some $S_4$ disjoint from $S_0$, requiring $S_4$ to be of the second form. Then $S_1$, $S_2$, $S_3$, and $S_4$ would all encircle the same crossing arc in the diagram, but $S_4$ would not glue to $S_1$. Hence additional normal squares would spiral around this arc, never closing up to form the annulus $A$. This is impossible.

So $S_3$ is of the third form, and when superimposed on $F$, $\bdy S_3$ is parallel to $\bdy S_0$. Recall, however, that $S_0$ and $S_3$ lie in different chunks.

\begin{figure}
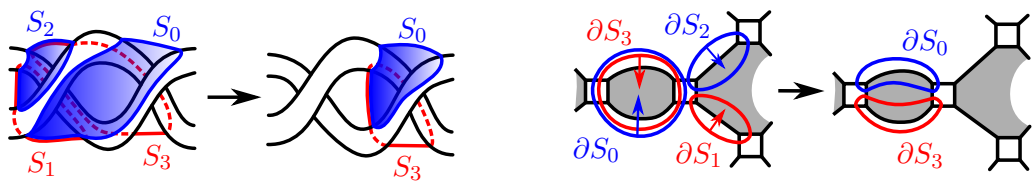
  \caption{An isotopy replaces $S_0, S_1, S_2, S_3$ with $S_0$ glued to $S_3$. Shown left in three dimensions, right in two.}
  \label{Fig:Isotopy3D}
\end{figure}

We can isotope $A$, removing squares $S_1$ and $S_2$, replacing $S_0$ and $S_3$ with normal squares of the second form. The isotopy is shown in three dimensions on the left of \reffig{Isotopy3D}; the image of the boundaries of the normal squares in the two chunks under the isotopy is shown on the right.
The isotopy strictly reduces the number of normal squares of $A$ of the third form. Thus repeating a finite number of times, we find that $A$ has no normal squares of the third form, and we are in a previous case. 
\end{proof}

\begin{corollary}\label{Cor:PrimeLink}
Let $\pi(L)$ be a reduced alternating diagram of a link $L$ on a generalised projection surface $F$ in a 3-manifold $Y$. Suppose that $r(\pi(L),F)\geq 4$ and $\hat{r}(\pi(L),F)>4$. Then the link is prime. That is, the exterior $X(L)$ contains no essential meridional annulus.
\end{corollary}

\begin{proof}
By \refthm{AnAnnularLink}, any essential annulus has slope parallel to the boundary slope of an annulus or M\"obius band formed between a string of bigons; this is not a meridian. 
\end{proof}


\subsection{Toroidal links}
Now we consider essential tori in link complements. As opposed to the case of regular alternating knots, essential tori do appear in our more general setting: Hayashi gives an example~\cite{hay95}. 
However, we can rule out essential tori under stronger hypotheses. The main results of this subsection are \refprop{Toroidal} and \refprop{HypToroidal}.

\begin{definition}\label{Def:MeridCompressible}
Let $Y$ be a compact orientable irreducible 3-manifold, and $L$ a link in $Y$. A closed incompressible surface $S$ embedded in the link exterior $X$ is \emph{meridionally compressible} if there is a disk $D$ embedded in $Y$ such that $D\cap S=\bdy D$ and the interior of $D$ intersects $L$ exactly once transversely. Otherwise $S$ is \emph{meridionally incompressible}.
\end{definition}

If an incompressible torus $T$ is parallel to a component of $L$ in $Y$, then $T$ is meridionally compressible. Hence to show that a torus $T$ is essential in $X$, it is sufficient to show that $T$ is incompressible and meridionally incompressible.

The following is proved using an argument due to Menasco in the alternating case~\cite{men84}. 

\begin{lemma}\label{Lem:Menasco}
Suppose $S$ is a closed normal surface in an angled chunk decomposition of a link $L$ that is reduced alternating on a generalised projection surface $F$ in $Y$ with $r(\pi(L),F)\geq 4$. Suppose further that one of the normal components of $S$ is a disk that is parallel into $F$. Then $S$ is meridionally compressible, and a meridional compressing disk meets $L$ at a crossing.
\end{lemma}

\begin{proof}
Let $S_i$ be a normal disk of $S$ that is parallel into $F$. Then $\bdy S_i$ bounds a disk in $F$, and there is an innermost normal disk $S_j$ of $S$ contained in the ball between $S_i$ and $F$ (using irreducibility of $Y$). The normal disk $S_j$ meets (interior) edges of the chunk decomposition, and these correspond to $\bdy S_j$ running adjacent to over-crossings of $\pi(L)$ on the side of $F$ containing the chunk.

In the case that $\pi(L)$ is checkerboard colourable, the boundary components of any face meeting $\pi(L)$ will be alternating under--over in a manner consistent with the orientation on a face. Then the curve $\bdy S_j$ must enter and exit the face by running adjacent to over-crossings on alternating sides; see \reffig{MenascoArg}, left. 
Crossing arcs opposite over-crossings are glued, hence $S$ intersects the opposite crossing arc as shown. Since $S_j$ is innermost, one of the arcs of $S\cap \bdy C$ opposite a crossing met by $\bdy S_j$ must be part of $\bdy S_j$.
Then $S$ encircles a meridian of $L$ at that crossing; see
\reffig{MenascoArg}, right.

\begin{figure}
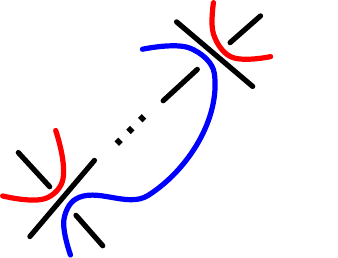
\hspace{1in}
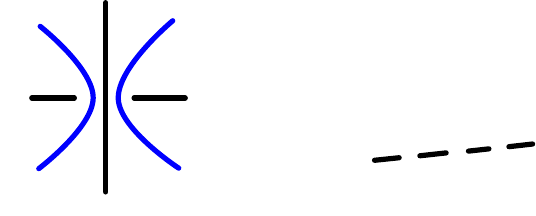
\caption{On the left, $\bdy S_j$ runs adjacent to crossings on opposite sides. The gluing of edges implies there are components of $S\cap \bdy C$ meeting the opposite crossing arc. Because $S_j$ is innermost, one of those must be on $\bdy S_j$, as shown on the right, which implies $S$ is meridionally compressible.}
\label{Fig:MenascoArg}
\end{figure}

Even in the case that $\pi(L)$ is not checkerboard colourable, we claim that $\bdy S_j$ must enter and exit a face in consistently alternating boundary components, meaning that $\bdy S_j$ enters and exits the face adjacent to crossings on opposite sides. Provided we can show this, the same argument as above will apply. If this does not hold, then $\bdy S_j$ must enter and exit each face in distinct boundary components, and the boundary components are all inconsistently oriented. But $\bdy S_j$ bounds a disk in $F$, so it meets each boundary component of each region an even number of times. There must be some outermost arc of intersection of $\bdy S_j$ with a boundary component of a face of the chunk. For this arc, $\bdy S_j$ must enter and exit the same boundary component of the same region, hence it enters and exits the face in a consistently alternating boundary component. 
\end{proof}

\begin{proposition}\label{Prop:Toroidal}
Let $\pi(L)$ be a reduced alternating diagram of a link $L$ on a generalised projection surface $F$ in a 3-manifold $Y$. 
Suppose further that $r(\pi(L),F)\geq 4$ and $\hat{r}(\pi(L),F)>4$.

If $Y\setminus N(F)$ is atoroidal but $X:=X(L)$ is toroidal, then any essential torus in $X$ is divided into normal annuli in the angled chunk decomposition of $X$, with each boundary component of a normal annulus completely contained in a single face of the chunk decomposition.
\end{proposition}

\begin{proof}
Suppose $T$ is an incompressible torus in $X$. Put $T$ into normal form with respect to the chunk decomposition. \refprop{GaussBonnet} implies that $T$ meets chunks in components of combinatorial area zero. \refprop{NonnegArea} implies each component has one of three forms. Since $Y\setminus N(F)$ is atoroidal, there are no incompressible tori in a chunk $C$ disjoint from $\bdy C$, so $T$ meets $\bdy C$. We will rule out the case that the torus is split into normal disks, and the only remaining case will be the desired conclusion.

Suppose a component of $T\cap C$ is a normal disk. Then it is glued to normal disks along its sides. Since $T$ is disjoint from $L$, all normal disks of $T\cap C$ are disjoint from boundary faces. By our assignment of exterior angles to interior edges, it follows that each normal disk of $T\cap C$ meets exactly four interior edges.

If any component of $T\cap C$ is parallel into $F$, then it bounds a disk $D$ in $F$. Then \reflem{Menasco} shows $T$ is meridionally compressible. Surger along a meridian compressing disk to obtain an annulus $A$. If the annulus is essential, it can be put into normal form and one of its boundary components is a meridian. But this contradicts \refthm{AnAnnularLink}: if there is an essential annulus, then $\pi(L)$ bounds a string of bigons and the slope of each boundary component of the annulus is integral, and hence cannot be meridional. So the annulus $A$ is not essential. But then $A$ must be parallel to a component of $\bdy N(L)$, which implies that $T$ is parallel to a component of $\bdy N(L)$, and therefore $T$ is not essential. 

So assume all disks are compressing disks for $F$ with nontrivial boundary on $F$. The disks alternate lying on either side of $F$. Each gives a compressing disk for $F$ with boundary meeting four edges, or meeting $\pi(L)$ four times. This is impossible, since $\hat{r}> 4$.
\end{proof}

By contrast to the previous result, if $Y\setminus N(F)$ is toroidal, then in the checkerboard colourable case, $X(L)$ will be toroidal as well.

\begin{proposition}\label{Prop:HypToroidal}
Let $\pi(L)$ be a weakly generalised alternating diagram of a link $L$ on a generalised projection surface $F$ in a 3-manifold $Y$. If $Y\setminus N(F)$ admits an embedded incompressible torus that is not parallel to $\bdy Y$ (but may be parallel to $\bdy N(F)$ in $Y\setminus N(F)$), then $X(L)$ is toroidal.
\end{proposition}

\begin{proof}
An embedded incompressible torus $T$ in $Y\setminus N(F)$ is also embedded in $Y\setminus L$. If it is not parallel into $\bdy Y$, we prove it is essential in the link exterior. For suppose $D$ is a compressing disk for $T$ in $Y\setminus L$. Consider the intersection of $D$ with the checkerboard surfaces $\Sigma$ and $\Sigma'$. Any innermost loop of intersection can be removed using \refthm{hress}. Then $D$ can be isotoped to be disjoint from both $\Sigma$ and $\Sigma'$, hence disjoint from $N(F)$. This contradicts the fact that $T$ is incompressible in $Y\setminus N(F)$.

If $E$ is a meridional compressing disk for $T$, then $E'=E\cap X$ is an annulus embedded in $X$ with one boundary component on $T$ and the other on $\bdy X$. Since $T$ is disjoint from $N(F)$, $T$ is disjoint from the checkerboard surfaces. But $E'$ can be isotoped so that the other boundary component has intersection number one with $\bdy\Sigma$, a contradiction since $E'\cap\Sigma$ consists of properly embedded arcs and loops.
\end{proof}

The proof of \refthm{Hyperbolic} now follows directly from previous results.

\begin{proof}[Proof of \refthm{Hyperbolic}]
Suppose first that $Y\setminus N(L)$ contains an essential annulus. If any component of $\bdy A$ lies on $\bdy N(L)$, then \refthm{AnAnnularLink} implies the diagram is a string of bigons on $F$. But because $F$ is not a 2-sphere and all regions are disks, this is impossible. 

So suppose that $Y\setminus N(L)$ contains an essential annulus $A$ with boundary components on $\bdy Y$. Put $A$ into normal form with respect to the angled chunk decomposition. If $A$ intersects $N(F)$, then an outermost sub-annulus must intersect the boundary of a chunk in a face that is not simply-connected by \refprop{NonnegArea}. But this is impossible since all the regions of $F\setminus N(L)$ are disks. Thus $A$ is disjoint from $N(F)$.

Conversely, an essential annulus $A\subset Y\setminus N(F)$ with $\bdy A\subset \bdy Y$ remains essential in $Y\setminus N(L)$: as in the proof of \refprop{HypToroidal} any compressing disk or boundary compressing disk can be isotoped to be disjoint from the checkerboard surfaces.

Now consider essential tori. If $Y\setminus N(F)$ is atoroidal, then 
by \refprop{Toroidal}, any essential torus in $X$ would have normal form meeting non-disk faces of the angled chunk decomposition of $X$. Since there are no such faces, $X(L)$ is atoroidal. Conversely, if $Y\setminus N(F)$ is toroidal, then \refprop{HypToroidal} implies $X(L)$ is toroidal. 

If $Y\setminus N(F)$ is atoroidal and $\bdy$-anannular, then $X(L)$ is irreducible, boundary irreducible, atoroidal and anannular. By work of Thurston, any manifold with torus boundary components that is irreducible, boundary irreducible, atoroidal and anannular has interior admitting a complete, finite volume hyperbolic structure; this is Thurston's hyperbolisation theorem~\cite{thu82}. It follows that $Y\setminus L$ is hyperbolic when $\bdy Y= \emptyset$ or when $\bdy Y$ consists of tori.

When $Y$ has boundary with genus higher than one, Thurston's hyperbolisation theorem still applies, as follows. Double $Y\setminus N(L)$ along components of $\bdy Y$ with genus greater than one, and denote the resulting manifold by $DY$. Then if $DY$ admits an essential sphere or disk, there must be an essential sphere or disk in $Y\setminus N(L)$ because of incompressibility of higher genus components of $\bdy Y$ in $Y\setminus N(L)$. Similarly, any essential torus or annulus in $DY$ gives rise to an essential torus or annulus in $Y\setminus N(L)$. But these have been ruled out. Thus $DY$ admits a complete, finite volume hyperbolic structure. By Mostow-Prasad rigidity, the involution of $DY$ that fixes the higher genus components of $\bdy Y$ must be realised by an isometry of $DY$ given by reflection in a totally geodesic surface. Cutting along the totally geodesic surface and discarding the reflection yields $Y\setminus N(L)$, now with a hyperbolic structure in which higher genus componets of $\bdy Y$ are totally geodesic. Thus there is a hyperbolic structure in this case as well. 
\end{proof}

In the case of knots, not links, we do not believe that $\hat{r}(\pi(L),F)>4$ should be required to rule out essential tori and annuli when all regions of $F\setminus\pi(L)$ are disks. Some restriction on representativity will be necessary, for Adams \emph{et al}~\cite{abb92} showed that the Whitehead double of the trefoil, which is a satellite knot and hence toroidal, has an alternating projection $\pi(K)$ onto a genus-2 Heegaard surface for $S^3$ with $r(\pi(K),F)=2$ and all regions disks. Even so, we conjecture the following in the case of knots.

\begin{conj}
Let $\pi(K)$ be a weakly generalised alternating diagram of a knot $K$ on a generalised projection surface $F$ in $Y$. Suppose $F$ is a Heegaard surface for $Y$ and all regions of $F\setminus\pi(K)$ are disks, and $Y\setminus N(F)$ is atoroidal and $\bdy$-anannular. Then $X(K)$ is hyperbolic.
\end{conj}


\subsection{The case of knots on a torus in the 3-sphere}
We can improve our results if we restrict to knots in $Y=S^3$ on $F$ a torus. Recall that a nontrivial knot $K$ in $S^3$ is either a satellite knot (the complement is toroidal), a torus knot (the complement is atoroidal but annular), or hyperbolic~\cite{thu82}. For a knot with a reduced alternating diagram $\pi(K)$ on $S^2$, Menasco~\cite{men84} proved $K$ is nontrivial, $K$ is a satellite knot if and only if $\pi(K)$ is not prime, and $K$ is a torus knot if and only if it has the obvious diagram of the $(p,2)$-torus knot. Otherwise $K$ is hyperbolic. 
Hyperbolicity has been studied for several other classes of knots and links in $S^3$; see for example \cite{ada03}, \cite{fkp15}.
Here, we completely classify the geometry of weakly generalised alternating knots on a torus in $S^3$.

\begin{theorem}[W.G.A. knots on a torus]\label{Thm:WGAKnotOnTorus}
Let $Y=S^3$, and let $F$ be a torus. Let $\pi(K)$ be a weakly generalised alternating projection of a knot $K\subset S^3$ onto $F$.
\begin{enumerate}
\item\label{Itm:NonHeegaard} If $F$ is not a Heegaard torus, then $K$ is a satellite knot.
\item\label{Itm:HeegaardAnnulus} If $F$ is a Heegaard torus and a region of $F\setminus\pi(K)$ is an annulus $A$,
  \begin{enumerate}
  \item\label{Itm:CoreNontriv} if the core of $A$ forms a nontrivial knot in $S^3$, then $K$ is a satellite knot;
  \item\label{Itm:CoreTriv} if the core of $A$ forms an unknot in $S^3$, then $K$ is hyperbolic.
  \end{enumerate}
\item\label{Itm:Disk} If $F$ is a Heegaard torus and all regions of $F\setminus\pi(K)$ are disks, then $K$ is hyperbolic.
\end{enumerate}
Moreover, items \refitm{NonHeegaard} and \refitm{CoreNontriv} also hold when $K$ is a link.
\end{theorem}

\begin{remark}
Items \refitm{NonHeegaard}, \refitm{CoreNontriv}, and \refitm{Disk} first appeared in \cite{how15t}. Item \refitm{CoreTriv} is new in this paper. We include the full proof for completeness. Note that \refthm{WGAKnotOnTorus} addresses all cases, since a diagram with multiple annular regions is the diagram of a multi-component link, and a diagram with a region which is a punctured torus does not satisfy the weakly prime or representativity conditions. 
\end{remark}

We will prove the theorem in a sequence of lemmas. First, we restrict essential annuli. 
The following theorem first appeared in \cite{how15t}.

\begin{theorem}\label{Thm:WGAKnotAnannular}
Let $Y=S^3$, and let $F$ be a generalised projection surface of positive genus. Let $\pi(K)$ be a weakly generalised alternating projection of a knot $K$ onto $F$. Then $K$ is not a torus knot.
\end{theorem}

\begin{proof}
One of the checkerboard surfaces is non-orientable, and from \refthm{hress} it is $\pi_1$-essential in $X$. But Moser~\cite{mos71} proved that the only 2-sided essential surfaces in a torus knot exterior are the Seifert surface of genus $\frac{1}{2}(p-1)(q-1)$ and the winding annulus at slope $pq$. The winding annulus covers a spanning surface if and only if $q=2$, in which case it covers a Mobius band at slope $pq$. The only way to obtain an alternating projection from these is a $(p,2)$-torus knot, which has a weakly generalised alternating projection onto $S^2$ only.
\end{proof}

\begin{lemma}\label{Lem:NonHeegaardTorus}
Let $\pi(L)$ be a weakly generalised alternating projection of a link $L$ onto a non-Heegaard torus $F$ in $S^3$. Then $L$ is a satellite link.
\end{lemma}

\begin{proof}
This follows by \refprop{HypToroidal}. Because $F$ is a non-Heegaard torus, there exists a torus $T$ which is parallel to $F$ and incompressible in $S^3\setminus N(F)$. 
\end{proof}

\begin{lemma}\label{Lem:CoreNontriv}
Let $\pi(L)$ be a weakly generalised alternating projection of a link $L$ onto a Heegaard torus $F$ in $S^3$, such that one region of $F\setminus\pi(L)$ is homeomorphic to an annulus $A$. If the core of $A$ forms a nontrivial knot in $S^3$, i.e.\ a $(p,q)$-torus knot, then $L$ is a satellite link on that $(p,q)$-torus knot.
\end{lemma}

\begin{proof}
Say the annulus $A$ is a subset of the checkerboard surface $\Sigma$. Consider the complementary annulus $A' = F\setminus A$ on $F$. The core of $A'$ is parallel to that of $A$, hence it forms a nontrivial knot in $S^3$. 
Let $T$ be the boundary of a neighbourhood of $A'$, chosen such that a solid torus $V$ bounded by $T$ contains $A'$ and contains $L$, and such that the core of $A$ lies on the opposite side of $T$. We claim $T$ is essential in $X$. It will follow that $L$ is a satellite of the core of $A'$, which is some $(p,q)$-torus knot isotopic to the core of $A$. 

If $D$ is a compressing disk for $T$, then $D$ cannot lie on the side of $T$ containing the core of $A$, since this side is a nontrivial knot exterior. So $D \subset V$. Then consider $D\cap\Sigma$ and $D\cap\Sigma'$, where $\Sigma'$ is the other checkerboard surface. Since $D$ is meridional in $V$, it is possible to isotope $D$ so that it intersects $\Sigma$ in exactly one essential arc $\beta$. All loops of intersection between $D$ and $\Sigma$ or $\Sigma'$ can be isotoped away using \refthm{hress}.
Since $\beta$ is disjoint from $\Sigma'$ and $\beta$ runs between the two distinct boundary components of $A$, it follows that $A\cup N(\beta\cap\Sigma)$ is homeomorphic to a once-punctured torus, contradicting the fact that $\pi(L)$ contains exactly one annular region. Thus $T$ is incompressible in $X$.

If $E$ is a meridional compressing disk for $T$, then $E'=E\cap X$ is an annulus embedded in $X$ with one boundary component on $T$ and the other on $\bdy X$. The boundary component on $T$ does not intersect $\Sigma'$, but the other boundary component has odd intersection number with $\Sigma'$ since $\Sigma'$ is a spanning surface. This is a contradiction.
\end{proof}

Continuing to restrict to the case that $F$ is a torus, if we further restrict to knots, we can rule out satellite knots in the remaining cases.

\begin{lemma}\label{Lem:AdamsNonsatellite}
Let $\pi(K)$ be a weakly generalised alternating diagram of a knot $K$ on a Heegaard torus $F$ in $S^3$ with all regions of $F\setminus \pi(K)$ disks. Then $K$ is not a satellite knot.
\end{lemma}

\begin{proof}
When $K$ is a knot, projected onto a Heegaard torus $F$, it is an example of a toroidally alternating knot, as in \cite{ada94}. Adams showed that a toroidally alternating diagram of a nontrivial prime knot is not a satellite knot \cite{ada94}; $K$ is nontrivial by \refcor{IrredBdryIrred} and prime by \refthm{wgaprime}.
\end{proof}

\begin{lemma}\label{Lem:CoreTriv}
Let $\pi(K)$ be a weakly generalised alternating diagram of a knot $K$ onto a Heegaard torus $F$ in $S^3$ such that a region of $F\setminus\pi(K)$ is homeomorphic to an annulus $A$. If the core of $A$ forms a trivial knot in $S^3$, then $K$ is not a satellite knot.
\end{lemma}

\begin{proof}
Suppose $T$ is an essential torus in $S^3\setminus K$. Isotope $T$ into normal form with respect to the chunk decomposition of $S^3\setminus K$. \refprop{GaussBonnet} (Gauss--Bonnet) implies that the combinatorial area of $T$ is zero. By \refprop{NonnegArea}, $T$ meets the chunks either in incompressible tori, or annuli, or disks.

If $T$ meets a chunk in an incompressible torus, then $T$ is completely contained on one side of $F$. But $F$ is a Heegaard torus, so each chunk is a solid torus. There are no incompressible tori embedded in a solid torus.

So suppose that a normal component of $T$ is an annulus $S$ meeting no edges of a chunk, with $\bdy S$ lying in non-contractible faces of the chunk. Because $F$ is a torus, $\bdy S$ lies in an annular face. Because $\pi(K)$ is connected ($K$ is a knot), there is only one annular face, namely $A$. Thus $\bdy S$ forms parallel essential curves in $A$.
Let $\alpha$ be the core of $A$. By hypothesis, $\alpha$ is the unknot, hence it is either a $(p,1)$ or $(1,q)$ torus knot, and we can choose a compressing disk $D$ for the projection torus $F$ such that $\bdy D$ meets $\alpha$ exactly once. Consider $D\cap T$. By incompressibility of $T$, any innermost loop of $D\cap T$ can be isotoped away. So $D\cap T$ consists of arcs with endpoints on $A\cap\bdy D$. Consider an outermost arc on $D$. This bounds a disk. We may use the disk to isotope $T$ through the chunk, removing the arc of intersection, and changing two adjacent curves of intersection of $T\cap A$ into a single loop of intersection, bounding a disk on $A$, which can then be isotoped away. Repeating, $T$ can be isotoped to remove all intersections with $A$, so $T$ can be isotoped to be disjoint from $F\times I$. Thus the annulus case does not occur.

Finally, suppose that $T$ is cut into normal disks meeting the chunks. Each has combinatorial area $0$. By definition of combinatorial area and the fact that each edge of the chunk is assigned angle $\pi/2$, each normal disk meets exactly four edges. 

First suppose a normal square $S$ of $T$ is parallel into $F$.
\reflem{Menasco} implies that $S$ is meridionally compressible with a meridional compressing disk meeting $K$ at a crossing. Surger along the meridional compressing disk to obtain a sphere meeting $K$ twice in exactly two meridians. 
This sphere bounds 3-balls on both sides in $S^3$, one of which contains only a trivial arc of $K$, since $K$ is prime by \refthm{wgaprime}. Depending on which $3$-ball contains the trivial arc, either $T$ is boundary parallel, or $T$ is compressible, both of which are contradictions. This is exactly Menasco's argument in \cite{men84}.

So each normal disk of $T$ is a compression disk for $F$. These normal disks cut the Heegaard tori into ball regions, each ball with boundary consisting of an annulus on $F$ and two disks of $T$. Because $T$ is separating, these can be coloured according to whether $K$ is inside or outside the ball. Consider the regions with $K$ outside. These are $I$-bundles over a disk. They are glued along regions of the chunk decomposition lying between a pair of edges running parallel to the fiber. Thus the regions are fibered, with the gluing preserving the fibering, and we obtain an $I$-bundle that is a submanifold of $S^3$ with boundary a single torus $T$. The only possibility is that the submanifold is an $I$-bundle over a Klein bottle, but no Klein bottle embeds in $S^3$. This contradiction is exactly Adams' contradiction in \cite{ada94}.
\end{proof}

%

\begin{proof}[Proof of \refthm{WGAKnotOnTorus}]
When $F$ is not a Heegaard torus, or when $F$ is a Heegaard torus but a region of $F\setminus\pi(K)$ is an annulus with knotted core, then Lemmas~\ref{Lem:NonHeegaardTorus} and~\ref{Lem:CoreNontriv}, respectively, imply that $K$ is a satellite knot.
In the remaining two cases, Lemmas~\ref{Lem:CoreTriv} and~\ref{Lem:AdamsNonsatellite} imply that $S^3\setminus K$ is atoroidal. Subsequently, \refthm{WGAKnotAnannular} implies that $S^3\setminus K$ is anannular. Because $S^3\setminus K$ is irreducible and boundary irreducible, it must be hyperbolic in these cases. 
\end{proof}

\begin{corollary}\label{Cor:GAKnotOnTorus}
  Let $\pi(K)$ be a generalised alternating diagram of a knot $K$ onto a torus $F$, as defined by Ozawa \cite{oza06}. Then $K$ is hyperbolic if and only if $F$ is Heegaard.
\end{corollary}

\begin{proof}
In Ozawa's definition of generalised alternating diagrams on $F$, the regions are disks, so the result follows from (1) and (3) of \refthm{WGAKnotOnTorus}.
\end{proof}

\section{Accidental, virtual fibered, and quasifuchsian surfaces}\label{Sec:Accidental}
We now switch to links with checkerboard colourable diagrams, and consider again the checkerboard surfaces $\Sigma$ and $\Sigma'$ of a weakly generalised alternating link diagram, which are essential by \refthm{hress}. Using the bounded angled chunk decompositions of the link exterior cut along $\Sigma$, we give further information about the surface $\Sigma$. In particular, we determine when it is accidental, quasifuchsian, or a virtual fiber.

We fix some notation. As before, we let $Y$ be a compact orientable irreducible 3-manifold with generalised projection surface $F$, such that if $\bdy Y\neq \emptyset$, then $\bdy Y$ consists of tori and $\bdy Y$ is incompressible in $Y\setminus N(F)$. Let $L$ be a link with a weakly generalised alternating diagram on $F$, $\Sigma$ a $\pi_1$-essential spanning surface, $\widetilde{\Sigma} = \bdy N(\Sigma)$, $M_{\Sigma} =  (Y\setminus L)\cut\Sigma$, and $P=\bdy M_{\Sigma} \cap \partial N(L)$ the parabolic locus.

\subsection{Accidental surfaces}

An \emph{accidental parabolic} is a non-trivial loop in $\Sigma$ which is freely homotopic into $\partial N(L)$ through $X:=Y\setminus N(L)$, but not freely homotopic into $\partial\Sigma$ through $\Sigma$.

Define an \emph{accidental annulus} for $\Sigma$ to be an essential annulus $A$ properly embedded in $M_{\Sigma}$ such that $\partial A=\alpha\cup\beta$ where $\beta\subset P$ and $\alpha\subset\widetilde{\Sigma}$ is an accidental parabolic. It is well-known that if spanning surface $\Sigma$ contains an accidental parabolic, then it admits an accidental annulus in $M_{\Sigma}$. See, for example Ozawa-Tsutsumi~\cite{ot03}, or \cite[Lemma~2.2]{fkp14}.

A surface $\Sigma$ is \emph{accidental} if $M_{\Sigma}$ contains an accidental annulus. 
If $\Sigma$ is accidental, then the slope of $\Sigma$ is the same as the slope of $\beta$ on $P$.

\begin{lemma}\label{Lem:squares}
Let $\pi(L)$ be a weakly generalised alternating link projection onto a generalised projection surface $F$ in $Y$, with shaded checkerboard surface $\Sigma$.
Let $A$ be an accidental annulus for $\Sigma$. Then $A$ decomposes as the union of an even number of normal squares, each with one side on a shaded face, one side on a boundary face, and two opposite sides on white faces. 
\end{lemma}

\begin{proof}
The annulus $A$ is essential, so it can be isotoped into normal form by \refthm{NormalForm}. Then by Gauss--Bonnet, \refprop{GaussBonnet}, the combinatorial area of $A$ is $a(A) = -2\pi\chi(A) + \pi\,p = 0+\pi\,p$, where $p$ is the number of times a boundary curve of $A$ runs from $\bdy N(\Sigma)$ to the parabolic locus and then back to $\bdy N(\Sigma)$. Since one component of $\bdy A$ is completely contained in $\bdy N(\Sigma)$ and the other is completely contained in the parabolic locus, it follows that $p=0$. Thus $a(A)=0$. Then $A$ is subdivided into normal pieces, each with combinatorial area $0$ within a chunk. By \refprop{NonnegArea}, each piece can be a disk or an annulus; the torus case does not arise because $A$ is an annulus. Since $\bdy A$ meets boundary faces of the chunk, there cannot be a normal component of $A$ meeting no edges. It follows that all normal components of $A$ are disks with combinatorial area $0$. Because each edge of the chunk has angle $\pi/2$, it follows that each piece meets exactly four edges, so it is a normal square.

Because $A$ runs through two distinct chunks, alternating between chunks, the total number of squares forming $A$ must be even.
Consider intersections of $A$ with white faces. We claim each arc of intersection runs from the component of $\bdy A$ on the parabolic locus $P$ to a shaded face. For if an arc runs from $\beta=\bdy A\cap P$ back to $\beta$, an innermost such arc along with an arc on $P$ bounds a disk in $A$. Because the white checkerboard surface $\Sigma'$ is boundary $\pi_1$-injective, we may push away this intersection. Similarly, if an innermost arc of intersection of $A$ with the white faces runs from $\alpha = \bdy A\cap \widetilde{\Sigma}$ back to $\alpha$ then it cuts off a bigon region of $A$ in normal form. This contradicts \refcor{NoNormalBigons}. So each arc of intersection of $A$ with a white face runs from $\beta$ to $\alpha$ on $A$. Thus each square is as described in the lemma. 
\end{proof}

\begin{lemma}\label{Lem:consecsquare}
Let $A$ be an accidental annulus for a checkerboard surface $\Sigma$ associated to a weakly generalised alternating diagram $\pi(L)$ on a generalised projection surface $F$ in a 3-manifold $Y$. Let $A_1, \dots, A_n$ be the decomposition of $A$ into normal squares, as in \reflem{squares}. Then no square $A_i$ can be parallel into $F$.
\end{lemma}

\begin{proof}
  Suppose by way of contradiction that $A_i$ is parallel into $F$. Then \reflem{Annulus} implies $L$ is a 2-component link, $\pi(L)$ is a string of bigons, $\Sigma$ is an annulus, and $\bdy A$ has a component parallel to the core of $\Sigma$. But then $\bdy A$ is freely homotopic into $\bdy N(L)$, contradicting the definition of an accidental annulus.
\end{proof}

Futer, Kalfagianni, and Purcell~\cite{fkp14} proved that a state surface associated to a $\sigma$-adequate $\sigma$-homogeneous link diagram has no accidental parabolics. This implies that the checkerboard surfaces associated to a reduced prime alternating diagram are not accidental in $X$. \reflem{consecsquare} leads to a new proof in that case.

\begin{corol}
Let $\pi(L)$ be a reduced prime alternating link projection onto $S^2$ in $S^3$; i.e.\ $L$ is alternating in the usual sense. Then neither checkerboard surface is accidental in $X$.
\end{corol}

\begin{proof}
Suppose $A$ is an accidental annulus for a checkerboard surface. Put $A$ into normal form. Every square making up $A$ must be parallel into $S^2$. This contradicts~\reflem{consecsquare}.
\end{proof}

\begin{theorem}\label{Thm:linkaccid}
Let $\pi(L)$ be a weakly generalised alternating diagram of a link $L$ on a generalised projection surface $F$ in a 3-manifold $Y$.  
Suppose $\hat{r}(\pi(L),F)> 4$. If $\Sigma$ is a checkerboard surface of $\pi(L)$, then $\Sigma$ is not accidental in $X$.
\end{theorem}

\begin{proof}
Suppose $A$ is an accidental annulus. The boundary of any normal square making up $A$ can be isotoped to meet $\pi(L)$ exactly four times. Since the representativity is strictly greater than four on one side of $F$, no square of $A$ on that side can be a compressing disk for $F$. Hence every square on one side of $F$ is parallel into $F$, contradicting~\reflem{consecsquare}.
\end{proof}

\reffig{linkaccideg} shows an example of a weakly generalised alternating link in $S^3$ where one of the checkerboard surfaces contains an accidental parabolic. This shows that we need the condition on representativity for links. However for knots, the condition is not necessary. The proof of the following is similar to \cite[Theorem~2.6]{fkp14}. 

\begin{figure}
    \centering
    \includegraphics{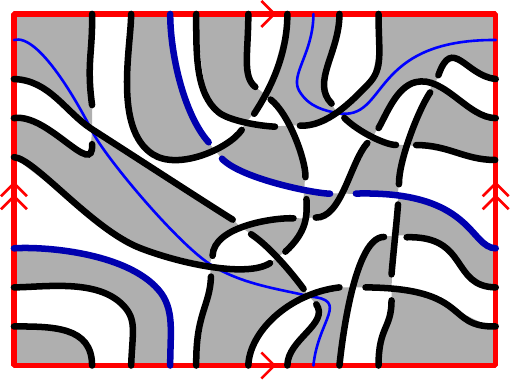}
    \caption{An accidental parabolic (thin line, in blue) on the shaded surface of a weakly generalised alternating link diagram. It is freely homotopic to the link component in darker blue. Note that the torus must be embedded in $S^3$ such that the identified edges each bound compressing disks for the torus.}
    \label{Fig:linkaccideg}
\end{figure}

\begin{theorem}\label{Thm:knotaccid}
Let $\pi(K)$ be a weakly generalised alternating knot projection. If $\Sigma$ is a checkerboard surface associated to $\pi(K)$ and $\Sigma'$ contains at least one disk region, then $\Sigma$ is not accidental in $X$.
\end{theorem}

\begin{proof}
Let $R'$ be a disk region of $\Sigma'$. Suppose that $A$ is an accidental annulus for $\Sigma$. Since $\bdy A\cap P$ must have the same slope as $\Sigma$, there must be one arc of $A\cap R'$ beginning on each segment of $\pi(L)$ in $\partial R'$. But it is impossible for these arcs to run to non-adjacent crossings without intersecting in the interior of $R'$. Hence there can be no accidental annulus for $\Sigma$.
\end{proof}


\subsection{Semi-fibers}

We say an essential surface $\Sigma$ in a compact, orientable 3-manifold $M$ is a \emph{semi-fiber} if either there is a fibration of $M$ over $S^1$ with fiber $\Sigma$, or if $\widetilde{\Sigma}$ is the common frontier of two twisted $I$-bundles over $\Sigma$ whose union is $M$.
Note that if $\Sigma$ is a semi-fiber, then $M\cut\Sigma$ is an $I$-bundle. In this section, we determine when a checkerboard surface is a semi-fiber in a weakly generalised alternating link complement. 

\begin{theorem}\label{Thm:Fibered}
Let $\pi(L)$ be a weakly generalised alternating diagram of a link $L$ on a generalised projection surface $F$ in a 3-manifold $Y$.
Let $\Sigma$ denote the shaded checkerboard surface, and suppose that the white regions of $F\setminus\pi(L)$ are disks. If the link exterior $X:=Y\setminus N(L)$ is hyperbolic, then $\Sigma$ is not a semi-fiber for $X$. Conversely, if $\Sigma$ is a semi-fiber for $X$, then the white surface $\Sigma'$ is an annulus or M\"obius band, and $\pi(L)$ is a string of white bigons bounding $\Sigma'$. 
\end{theorem}

The proof of \refthm{Fibered} will follow from the following lemma, which describes how an $I$-bundle embedded in $M_\Sigma=X\cut\Sigma$ meets white regions. It is essentially \cite[Lemma~4.17]{fkp13}.

\begin{lemma}\label{Lem:ProductRectangles}
Let $\pi(L)$ and $\Sigma$ be as in the statement of \refthm{Fibered}, so white regions of $F\setminus \pi(L)$ are disks. 
Let $B$ be an $I$-bundle embedded in $M_\Sigma = X\cut \Sigma$, with horizontal boundary on $\widetilde{\Sigma}$ and essential vertical boundary.
Let $R$ be a white region of $F\setminus\pi(L)$. Then $B\cap R$ is a product rectangle $\alpha\times I$, where $\alpha\times\{0\}$ and $\alpha\times\{1\}$ are arcs of ideal edges of $R$. 
\end{lemma}

\begin{proof}
First suppose $B=Q\times I$ is a product $I$-bundle over an orientable base. 
Consider a component of $\bdy(B\cap R)$. If it lies entirely in the interior of $R$, then it lies in the vertical boundary $V=\bdy Q\times I$.
The intersection $V\cap R$ then contains a closed curve component; an innermost one bounds a disk in $R$. Since the vertical boundary is essential, we may isotope $B$ to remove these intersections. So assume each component of $\bdy(B\cap R)$ meets $\widetilde{\Sigma}$.

Note $R\cap \widetilde{\Sigma}$ consists of ideal edges on the boundary of the face $R$. It follows that the boundary of each component of $B\cap R$ consists of arcs $\alpha_1$, $\beta_1$, $\dots$, $\alpha_n$, $\beta_n$ with $\alpha_i$ an arc in an ideal edge of $R\cap \widetilde{\Sigma}$ and $\beta_i$ in the vertical boundary of $B$, in the interior of $R$.
We may assume that each arc $\beta_i$ runs between distinct ideal edges, else isotope $B$ through the disk bounded by $\beta_i$ and an ideal edge to remove $\beta_i$, and merge $\alpha_i$ and $\alpha_{i+1}$.

We may assume that $\beta_i$ runs from $Q\times\{0\}$ to $Q\times\{1\}$, for if not, then $\beta_i \subset R$ is an arc from $\bdy Q\times\{1\}$ to $\bdy Q\times\{1\}$, say, in an annulus component of $\bdy Q \times I$. Such an arc bounds a disk in $\bdy Q\times I$. This disk has boundary consisting of the arc $\beta_i$ in $R$ and an arc on $\bdy Q\times\{1\} \subset\widetilde{\Sigma}$. If the disk were essential, it would give a contradiction to \refcor{NoNormalBigons}. So it is inessential, and we may isotope $B$ to remove $\beta_i$, merging $\alpha_i$ and $\alpha_{i+1}$. It now follows that $n$ is even.

Finally we show that $n=2$, i.e.\ that each component of $B\cap R$ is a quadrilateral with arcs $\alpha_1, \beta_1, \alpha_2, \beta_2$. For if not, there is an arc $\gamma\subset R$ with endpoints on $\alpha_1$ and $\alpha_3$. By sliding along the disk $R$, we may isotope $B$ so $\gamma$ lies in $B\cap R$. Then note that $\gamma$ lies in $Q\times I$ with endpoints on $Q\times\{1\}$. It must be parallel vertically to an arc $\delta\subset Q\times\{1\} \subset \widetilde{\Sigma}$. This gives another disk $D$ with boundary consisting of an arc on $R$ and an arc on $\widetilde{\Sigma}$. 

The disk $D$ cannot be essential since that would contradict \refcor{NoNormalBigons}. Hence $D$ is inessential which means that $\alpha_1$ and $\alpha_3$ lie on the same ideal edge of $R$. From above we know that $\alpha_2$ must lie on a different ideal edge of $R$. But it is impossible to then connect up the endpoints of the $\alpha_i$ to enclose a subdisk of $R$.

It follows that each component of $B\cap R$ is a product rectangle $\alpha\times I$ where $\alpha\times\{1\} = \alpha_1$ is an arc of an ideal edge of $W$ and $\alpha\times\{0\} = \alpha_2$ is an arc of an ideal edge of $W$.

Next suppose $B$ is a twisted $I$-bundle $B=Q \widetilde{\times} I$ where $Q$ is non-orientable. Let $\gamma_1, \dots, \gamma_m$ be a maximal collection of orientation reversing closed curves on $Q$. Let $A_i\subset B$ be the $I$-bundle over $\gamma_i$. Each $A_i$ is a M\"obius band. The bundle $B_0 = B\setminus (\cup_i A_i)$ is then a product bundle $B_0 = Q_0\times I$ where $Q_0=Q\setminus (\cup_i \gamma_i)$ is an orientable surface. Our work above then implies that $B_0\cap R$ is a product rectangle for each white region $R$. To obtain $B\cap R$, we attach the vertical boundary of such a product rectangle to the vertical boundary of a product rectangle of $A_i$. This procedure respects the product structure of all rectangles, hence the result is a product rectangle. 
\end{proof}

\begin{proof}[Proof of Theorem~\ref{Thm:Fibered}]
If $\Sigma$ is a semi-fiber for $X=Y\setminus N(L)$, then the manifold $M_\Sigma = X\cut \Sigma$ is an $I$-bundle. Lemma~\ref{Lem:ProductRectangles} implies $M_\Sigma$ intersects each white face $R$ in a product rectangle of the form $\alpha\times I$, where $\alpha\times\{0\}$ and $\alpha\times\{1\}$ lie on ideal edges of $R$. Since $R \subset M_\Sigma$, the face $R$ is a product rectangle, with exactly two ideal edges $\alpha\times\{0\}$ and $\alpha\times\{1\}$. Thus $R$ is a bigon.

If every white face $R$ is a bigon, then $\pi(L)$ is a string of bigons lined up end to end on $F$. But then the white checkerboard surface $\Sigma'$ is made up of the interior of the bigons, hence it is a M\"obius band or annulus. Then the white surface $\Sigma'$ is $\pi_1$-essential by \refthm{hress}; it follows that $X$ contains an essential annulus and therefore is not hyperbolic.
\end{proof}


\subsection{Quasifuchsian surfaces}

A properly embedded $\pi_1$-essential surface $\Sigma$ in a hyperbolic $3$-manifold is quasifuchsian if the lift of $\Sigma$ to $\mathbb{H}^3$ is a plane whose limit set is a Jordan curve on the sphere at infinity.
The following theorem follows from work of Thurston~\cite{thu79} and  Bonahon~\cite{bon86}; see also Canary--Epstein--Green~\cite{ceg87}. 

\begin{theorem}\label{Thm:qftrichot}
Let $\Sigma$ be a properly embedded $\pi_1$-essential surface in a hyperbolic $3$-manifold of finite volume. Then $\Sigma$ is exactly one of: quasifuchsian, semi-fibered, or accidental.
\end{theorem}

Fenley~\cite{fen98} showed that a minimal genus Seifert surface for a non-fibered hyperbolic knot is quasifuchsian. 
Futer, Kalfagianni, and Purcell~\cite{fkp14} showed that the state surface associated to a homogeneously adequate knot diagram is quasifuchsian whenever such a knot is hyperbolic. This includes the case of the checkerboard surfaces associated to a reduced planar alternating diagram of a hyperbolic knot in $S^3$.

\begin{corol}\label{Cor:Quasifuchsian}
Let $\pi(L)$ be a weakly generalised alternating diagram of a link $L$ on a generalised projection surface $F$ in a 3-manifold $Y$.
Suppose that $F$ has genus at least $1$, and that $Y\setminus N(F)$ is atoroidal and $\bdy$-anannular. Suppose further that $\hat{r}(\pi(L),F)>4$, and that the regions of $F\setminus \pi(L)$ are disks. If $\Sigma$ is a checkerboard surface associated to $\pi(L)$, then $\Sigma$ is quasifuchsian. 
\end{corol}

\begin{proof}
\refthm{Hyperbolic} ensures that $Y\setminus L$ is hyperbolic, so the surface $\Sigma$ cannot be a virtual fiber by \refthm{Fibered}. It cannot be accidental by \refthm{linkaccid}. By \refthm{qftrichot}, the only remaining possibility is for $\Sigma$ to be quasifuchsian.
\end{proof}

\section{Bounds on volume}\label{Sec:Volume}
In this section, we give a lower bound on the volumes of hyperbolic weakly generalised alternating links in terms of their diagrams. Lackenby was the first to bound volumes of alternating knots in terms of a diagram \cite{lac04}. Our method of proof is similar to his, using angled chunk decompositions rather than ideal polyhedra. Much of his work goes through.

\begin{definition}\label{Def:Guts}
Let $S$ be a $\pi_1$-essential surface properly embedded in an orientable hyperbolic 3-manifold $M$. Consider $M\cut S$. This is a 3-manifold with boundary; therefore it admits a JSJ-decomposition, decomposing it along essential tori and annuli into $I$-bundles, Seifert fibered solid tori, and \emph{guts}. The \emph{guts}, denoted $\guts(M\cut S)$, is the portion of the manifold $M\cut S$ after the JSJ-decomposition that admits a hyperbolic metric with geodesic boundary. 
\end{definition}

To bound the volume, we will use the following theorem applied to checkerboard surfaces.

\begin{theorem}[Agol-Storm-Thurston \cite{ast07}]\label{Thm:agolguts}
Let $M$ be an orientable, finite volume hyperbolic $3$-manifold. We allow $M$ to have nonempty boundary, but in that case, we require that every component of $\bdy M$ is totally geodesic in the hyperbolic structure. (Note $M$ may also have cusps, but these are not part of $\bdy M$.)
Finally, let $S$ be a $\pi_1$-essential surface properly embedded in $M$ disjoint from $\bdy M$. 
Then
\[\vol(M)\geq-v_8\chi(\guts(M\cut S)),\]
where $v_8 =3.66\dots$ is the volume of a regular ideal octahedron.
\end{theorem}

\begin{proof}
This follows from \cite[Theorem~9.1]{ast07} as follows. If $\bdy M$ is empty, which includes the case that $M$ has cusps, then the statement is equivalent to that of \cite[Theorem~9.1]{ast07}.

Otherwise, consider the double $DM$ of $M$, doubling over components of $\bdy M$. Because $\bdy M$ is incompressible, and $M$ is irreducible, anannular and atoroidal (by hyperbolicity of $M$), $DM$ is a hyperbolic 3-manifold of finite volume. Let $DS$ denote the union of $S$ and its double in $DM$. Then $DS$ is $\pi_1$-essential in $DM$, using the fact that $\bdy M$ is incompressible in $M$. Finally, let $\Sigma$ denote the union of $DS$ and the components of $\bdy M$. Then $\Sigma$ is $\pi_1$-essential in $DM$. 
Thus \cite[Theorem~9.1]{ast07} implies
\[ 2\vol(M) = \vol(DM) \geq -v_8\chi(\guts(DM\cut \Sigma)). \]
Note that $DM\cut \Sigma$ is exactly two copies of $M\cut S$. Thus
\[ \chi(\guts(DM\cut\Sigma)) = 2\chi(\guts(M\cut S)). \qedhere \]
\end{proof}

\begin{definition}\label{Def:WeaklyTwistReduced}
A reduced alternating diagram $\pi(L)$ on a generalised projection surface $F$ is said to be \emph{weakly twist reduced} if the following holds. 
Suppose $D$ is a disk in $F$ with $\bdy D$ meeting $\pi(L)$ transversely in four points, adjacent to exactly two crossings. Then either $D$ contains only bigon faces of $F\setminus\pi(L)$, oriented such that the two crossings adjacent to $D$ belong to bigon faces, or $F\setminus D$ contains a disk $D'$, where $\bdy D'$ meets $\pi(L)$ adjacent to the same two crossings and bounds only bigon faces, again with bigons including the two crossings. See \reffig{WeaklyTwistReduced}.
\end{definition}

\begin{figure}
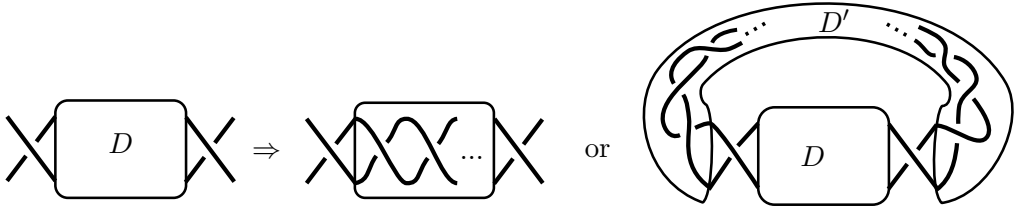
\caption{Weakly twist reduced: If the boundary of a disk $D\subset F$ meets the diagram in four points adjacent to exactly two crossings, then either $D$ bounds a string of bigons, or there is a disk outside $D$ meeting the diagram adjacent to the same two crossings, bounding a string of bigons.}
\label{Fig:WeaklyTwistReduced}
\end{figure}

\begin{definition}\label{Def:TwistRegion}
Let $\pi(L)$ be a weakly generalised alternating diagram on a surface $F$. A \emph{twist region} of $\pi(L)$ is either
\begin{itemize}
\item a string of bigon regions of $\pi(L)$ arranged vertex to vertex that is maximal in the sense that no larger string of bigons contains it, or
\item a single crossing adjacent to no bigons.
\end{itemize}
The \emph{twist number} $\tw(\pi(L))$ is the number of twist regions in a weakly twist-reduced diagram.
\end{definition}

Any diagram of a weakly generalised alternating link on a generalised projection surface $F$ can be modified to be weakly twist reduced. For if $D$ is a disk as in the definition and neither $D$ nor $D'$ contains only bigons, then a flype in $D$ on $F$ moves the crossings met by $\bdy D$ to be adjacent, either canceling the crossings or leaving a bigon between, reducing the number of twist regions on $F$.

\begin{lemma}\label{Lem:MarcLemma7}
Let $\pi(L)$ be a reduced alternating diagram on $F$ in $Y$, such that all regions of $F\setminus\pi(L)$ are disks. 
Let $D_1$ and $D_2$ be normal disks parallel to $F$ such that $\bdy D_1$ and $\bdy D_2$ meet exactly four interior edges. Isotope $\bdy D_1$ and $\bdy D_2$ to minimise intersections $\bdy D_1\cap\bdy D_2$ in faces. If $\bdy D_1$ intersects $\bdy D_2$ in a face of the chunk, then $\bdy D_1$ intersects $\bdy D_2$ exactly twice, in faces of the same colour. 
\end{lemma}

\begin{proof}
The boundaries $\bdy D_1$ and $\bdy D_2$ are quadrilaterals, with sides of $\bdy D_i$ between intersections with interior edges. Note that $\bdy D_1$ can intersect $\bdy D_2$ at most once in any of its sides by the requirement that the number of intersections be minimal (else isotope through a disk face). 
Thus there are at most four intersections of $\bdy D_1$ and $\bdy D_2$. If $\bdy D_1$ meets $\bdy D_2$ four times, then the two quads run through the same regions, both bounding disks, and can be isotoped off each other using the fact that the diagram is weakly prime. Since the quads intersect an even number of times, there are either zero or two intersections. If zero intersections, we are done.

So suppose there are two intersections. Since all regions are disks, $\pi(L)$ is checkerboard colourable. Suppose $\bdy D_1$ intersects $\bdy D_2$ exactly twice in faces of the opposite colour. Then an arc $\alpha_1\subset \bdy D_1$ has both endpoints on $\bdy D_1\cap \bdy D_2$ and meets only one intersection of $\bdy D_1$ with an interior edge of the chunk decomposition. Similarly, an arc $\alpha_2\subset\bdy D_2$ has both endpoints on $\bdy D_1\cap \bdy D_2$ and meets only one intersection of $\bdy D_2$ with an interior edge of the chunk decomposition. Then $\alpha_1\cup \alpha_2$ is a closed curve on $F$ meeting exactly two interior edges of the chunk decomposition and bounding a disk in $F$. 

There are four cases: $\alpha_1\cup\alpha_2$ bounds the disk $D_1\cap D_2$, $D_1\setminus D_2$, $D_2\setminus D_1$, or $D_1\cup D_2$ in the chunk. It follows that the corresponding disk on $F$ contains no crossings. But then the arcs $\alpha_1$ and $\alpha_2$ are parallel and in the first three cases, the arcs can be isotoped through the disk to remove the intersections of $D_1$ and $D_2$. In the final case, the interior of the disk bounded by $\alpha_1\cup\alpha_2$ is not disjoint from $\bdy D_1\cup \bdy D_2$, but it is still possible to isotope $D_1$ to be disjoint from $D_2$. 
\end{proof}

\begin{theorem}\label{Thm:chiguts}
Let $\pi(L)$ be a weakly twist-reduced, weakly generalised alternating diagram on a generalised projection surface $F$ in a 3-manifold $Y$. 
If $\bdy Y \neq \emptyset$, then suppose $\bdy Y$ is incompressible in $Y\setminus N(F)$. Also suppose $Y\setminus N(F)$ is atoroidal and $\bdy$-anannular.
Suppose $F$ has genus at least one, all regions of $F\setminus\pi(L)$ are disks, and $r(\pi(L),F)>4$.

Let $S$ and $W$ denote the checkerboard surfaces of $\pi(L)$, and let $r_S = r_S(\pi(L))$ and $r_W = r_W(\pi(L))$ denote the number of non-bigon regions of $S$ and $W$ respectively.
Write $M_S = X\cut S$, and $M_W=X\cut W$. Then
\[\chi(\guts(M_S))=\chi(F)+\frac{1}{2}\chi(\bdy Y)-r_W,\hspace{5mm} \chi(\guts(M_W))=\chi(F)+\frac{1}{2}\chi(\bdy Y)-r_S.\]
\end{theorem}

We know $M_S$ is obtained by gluing chunks along white faces only.
If $\guts(M_S)$ happens to equal $M_S$, then $\chi(M_S)$ is obtained by taking the sum of the Euler characteristics of each chunk $C$, which is $\chi(C) = \chi(\bdy C)/2$, and subtracting one from the total for each (disk) white face. Note that each component of $F$ appears as a boundary component of exactly two chunks; the other boundary components come from $\bdy Y$. Thus the sum of Euler characteristics of chunks is $\chi(F) + \chi(\bdy Y)/2$.

However, note that if there are any white bigon regions, then the white bigon can be viewed as a quad with two sides on the parabolic locus $P$ and two sides on $\widetilde{S}$. A neighbourhood of this is part of an $I$-bundle, with horizontal boundary on $\widetilde{S}$ and vertical boundary on the parabolic locus. Thus white bigon faces cannot be part of the guts. We will remove them.

To find any other $I$-bundle or Seifert fibered components, we must identify any essential annuli embedded in $M_S$, disjoint from the parabolic locus, and with $\bdy A\subset \widetilde{S}$; these give the JSJ-decomposition of $M_S$. So suppose $A$ is such an essential annulus.

\begin{definition}\label{Def:ParabolicallyCompressible}
We say that $A$ is \emph{parabolically compressible} if there exists a disk $D$ with interior disjoint from $A$, with $\bdy D$ meeting $A$ in an essential arc $\alpha$ on $A$, and with $\bdy D\setminus \alpha$ lying on $\widetilde{S}\cup P$, with $\alpha$ meeting $P$ transversely exactly once. We may surger along such a disk; this is called a \emph{parabolic compression}, and it turns the annulus $A$ into a disk meeting $P$ transversely exactly twice, with boundary otherwise on $\widetilde{S}$. This is called an \emph{essential product disk (EPD)}; each EPD will be part of the $I$-bundle component of the JSJ-decomposition.
\end{definition}

\begin{lemma}\label{Lem:NoParabComprAnnulus} 
Let $\pi(L)$ be a weakly twist-reduced, weakly generalised alternating diagram on a generalised projection surface $F$ in a 3-manifold $Y$. Suppose
all regions of $F\setminus\pi(L)$ are disks, $r(\pi(L),F)>4$, and there are no white bigon regions. Let $A$ be an essential annulus embedded in $M_S = X\cut S$, disjoint from the parabolic locus, with $\bdy A\subset \widetilde{S}$. Then $A$ is not parabolically compressible.
\end{lemma}

\begin{proof}
Suppose $A$ is parabolically compressible. Then surger along a parabolic compressing disk to obtain an EPD $E$. Put $E$ into normal form with respect to the chunk decomposition of $M_S$. If $E$ intersects a white (interior) face coming from $W$, then consider the arcs $E\cap W$. Such an arc has both endpoints on $\widetilde{S}$. If one cuts off a disk on $E$ that does not meet the parabolic locus, then there will be an innermost such disk. This will be a normal bigon, i.e.\ a disk in the chunk decomposition meeting exactly two surface edges. 
This contradicts \refcor{NoNormalBigons}. 
So $E\cap W$ consists of arcs running from $\widetilde{S}$ to $\widetilde{S}$, cutting off a disk meeting the parabolic locus $P$ on either side. Thus $W$ cuts $E$ into normal squares $\{E_1, \dots, E_n\}$. On the end of $E$, the square $E_1$ has one side on $W$, two sides on $\widetilde{S}$, and the final side on a boundary face.
We may isotope slightly off the boundary face into an adjacent white face so that $E_1$ remains normal. Then $E_1$ and $E_2$ are both squares meeting no boundary faces.

Since $r(\pi(L),F)>4$, no $E_i$ can be a compressing disk for $F$, hence each is parallel into $F$. Superimpose $E_1$ and $E_2$ onto the boundary of one of the chunks. An edge of $E_1$ in a white face $U\subset W$ is glued to an edge of $E_2$ in the same white face. Because the two are in different chunks, when we superimpose, $\bdy E_2\cap U$ is obtained from $\bdy E_1\cap U$ by a rotation in $U$.

If $E_1\cap U$ is not parallel to a single boundary edge, then $\bdy E_1\cap U$ and $\bdy E_2\cap U$ intersect; see \reffig{NoEPD}.
But then \reflem{MarcLemma7} implies that $\bdy E_1$ and $\bdy E_2$ also intersect in another white face. But $\bdy E_1$ is parallel to a single boundary edge in its second white face, so $\bdy E_2$ cannot intersect it. This is a contradiction. 

\begin{figure}
  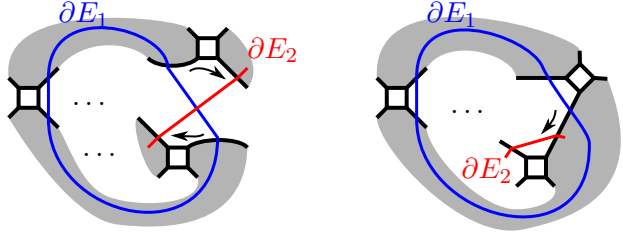
  \caption{Left: $\bdy E_1$ is not parallel to a boundary edge in $U$, hence $\bdy E_1$ meets $\bdy E_2$ in $U$. Right: $\bdy E_1$ is parallel to a boundary edge. }
  \label{Fig:NoEPD}
\end{figure}

So $E_1\cap U$ is parallel to a single boundary edge (and hence so is $E_2\cap U$). But then $E_1$ meets both white faces in arcs parallel to boundary edges. Isotoping $\bdy E_1$ slightly, this gives a closed curve on $F$ meeting $\pi(L)$ in crossings. If the crossings are distinct, then because $\pi(L)$ is weakly twist reduced, the two crossings must bound white bigons between them, contradicting the fact that there are no white bigon regions in the diagram. If the crossings are not distinct, then $\bdy E_1$ encircles a single crossing of $\pi(L)$. Repeating the argument with $E_2$ and $E_3$, and so on, we find that each $\bdy E_i$ encircles a single crossing, and thus the original annulus $A$ is boundary parallel. This contradicts the fact that $A$ is essential.

So if there is an EPD $E$, it cannot meet $W$. Then it lies completely on one side of $F$. Its boundary runs through two shaded faces and two boundary faces. Slide slightly off the boundary faces; we see that its boundary defines a curve meeting the knot four times. Because $r(\pi(L),F)>4$, it cannot be a compressing disk, so $E$ is parallel into $F$. But then weakly twist-reduced implies its boundary encloses a string of white bigons, or its exterior in $F$ contains a string of white bigons meeting $\bdy E$ in the same two crossings. Since there are no white bigons, $\bdy E$ is boundary parallel, and the EPD is not essential. 
\end{proof}

\begin{lemma}\label{Lem:NoParabIncomp}
Let $\pi(L)$ be a weakly twist-reduced, weakly generalised alternating diagram on a generalised projection surface $F$ in a 3-manifold $Y$. 
Suppose $F$ has genus at least one, all regions of $F\setminus\pi(L)$ are disks, $r(\pi(L),F)>4$, and there are no white bigon regions. Let $A$ be an essential annulus embedded in $M_S = X\cut S$, disjoint from the parabolic locus, with $\bdy A\subset \widetilde{S}$. Then $A$ cannot be parabolically incompressible.
\end{lemma}

\begin{proof}
Suppose $A$ is parabolically incompressible and put $A$ into normal form; \refprop{NonnegArea} implies $W$ cuts $A$ into squares $E_1, \dots, E_n$. Representativity $r(\pi(L),F)>4$ implies each square is parallel into $F$. Note that if a component of intersection $E_i\cap W$ is parallel to a boundary edge, then the disk of $W$ bounded by $E_i\cap W$, the boundary edge, and $\widetilde{S}$ defines a parabolic compression disk for $A$, contradicting the fact that $A$ is parabolically incompressible. So each component of $E_i\cap W$ cannot be parallel to a boundary edge.

Again superimpose all squares $E_1, \dots, E_n$ on one of the chunks. The squares are glued in white faces, and cut off more than a single boundary edge in each white face, so $\bdy E_i$ must intersect $\bdy E_{i+1}$ in a white face; see again \reffig{NoEPD}. Then \reflem{MarcLemma7} implies $\bdy E_i$ intersects $\bdy E_{i+1}$ in both of the white faces it meets. Similarly, $\bdy E_i$ intersects $\bdy E_{i-1}$ in both its white faces. Because $E_{i-1}$ and $E_{i+1}$ lie in the same chunk, they are disjoint (or $E_{i-1} = E_{i+1}$, but this makes $A$ a M\"obius band rather than an annulus).
This is possible only if $E_{i-1}$, $E_i$, and $E_{i+1}$ line up as in \reffig{FusedUnits} left, bounding tangles as shown. Lackenby calls such tangles \emph{units} \cite{lac04}. Then all $E_j$ form a cycle of such tangles, as in \reffig{FusedUnits} right.

\begin{figure}
  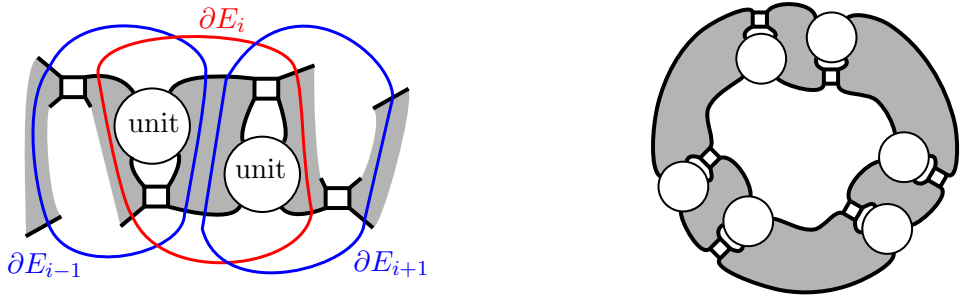
  \caption{Left: $E_{i-1}$, $E_i$, and $E_{i+1}$ must intersect as shown. Right: cycle of three such tangles.}
  \label{Fig:FusedUnits}
\end{figure}

Since all white regions are disks, the inside and outside of the cycle are disks. Also, all $E_i$ bound disks on $F$. But then $F$ is the union of two disks and an annulus made up of disks; this is a sphere. This contradicts the fact that the genus of $F$ is at least 1.
\end{proof}

\begin{proof}[Proof of \refthm{chiguts}]
First we rule out essential annuli with boundary components on $\bdy Y$. Suppose $A$ is an essential annulus in $M_S$ with both boundary components on $\bdy Y$. Because $Y\setminus N(F)$ is $\bdy$-anannular by assumption, $A$ must meet $N(F)$, and hence it meets a white face of the chunk decomposition. When we put $A$ into normal form, it decomposes into pieces, each with combinatorial area zero. The fact that all regions of $F\setminus \pi(L)$ are disks along with \refprop{NonnegArea} imply that each normal piece is a disk meeting four interior edges. But then $A$ intersects $\widetilde{S}$, hence is not embedded in $M_S$. Now suppose $A$ is an essential annulus in $M_S$ with only one boundary component on $\bdy Y$ and the other on $\widetilde{S}$. Because regions of $F\setminus\pi(L)$ are disks, the component of $\bdy Y$ on $\widetilde{S}$ must run through both white and shaded faces, and thus again $A$ is decomposed into squares in the chunk decomposition, each with two sides on $\widetilde{S}$. But then $A$ cannot have one boundary component on $\bdy Y$.

Now we rule out essential annuli with both boundary components on $\widetilde{S}$. Suppose first that $\pi(L)$ has no white bigon regions. Then \reflem{NoParabComprAnnulus} and \reflem{NoParabIncomp} imply that there are no embedded essential annuli in $M_S$ with both boundary components on $\widetilde{S}$.
It follows that $\chi(\guts(M_S)) = \chi(M_S) = \chi(F)+\chi(\bdy Y)/2-r_W$.

If $\pi(L)$ contains white bigon regions, then
replace each string of white bigon regions in the diagram $\pi(L)$ by a single crossing, to obtain a new link. Since white bigons lead to product regions, the guts of the shaded surface of the new link agrees with $\guts(M_S)$. Hence  $\chi(\guts(M_S)) = \chi(M_S) = \chi(F)+\chi(\bdy Y)/2-r_W$ in this case.

An identical argument applies to $M_W$, replacing the roles of $W$ and $S$ above.
\end{proof}

\begin{theorem}\label{Thm:volguts}
Let $\pi(L)$ be a weakly twist-reduced, weakly generalised alternating diagram on a generalised projection surface $F$ in a 3-manifold $Y$.
Suppose 
that if $\bdy Y\neq \emptyset$, then $\bdy Y$ is incompressible in $Y\setminus N(F)$. Suppose also that $Y\setminus N(F)$ is atoroidal and $\bdy$-anannular. Finally suppose $F$ has genus at least one, that all regions of $F\setminus\pi(L)$ are disks, and that $r(\pi(L),F)>4$. 
Then $Y\setminus L$ is hyperbolic and 
\[\vol(Y\setminus L)\geq \frac{v_8}{2}\left(\tw(\pi(L))-\chi(F) -
\chi(\bdy Y)\right).\]
\end{theorem}

\begin{proof}
The fact that $Y\setminus L$ is hyperbolic follows from \refthm{Hyperbolic}.

Let $\Gamma$ be the $4$-regular graph associated to $\pi(L)$ by replacing each twist-region with a vertex.
Let $|v(\Gamma)|$ denote the number of vertices of $\Gamma$, $|f(\Gamma)|$ the number of regions. Because $\Gamma$ is 4-valent, the number of edges is $2|v(\Gamma)|$, so
\[\chi(F)=-|v(\Gamma)|+|f(\Gamma)|=-\tw(\pi(L))+r_S+r_W.\]
Then applying Theorems~\ref{Thm:agolguts} and~\ref{Thm:chiguts} gives
\begin{align*}
\vol(X) \ &\geq \ -\frac{1}{2}v_8\chi(\guts(M_S))-\frac{1}{2}v_8\chi(\guts(M_W)) \\
\ &= \ -\frac{1}{2}v_8(2\chi(F) + \chi(\bdy Y) -r_S-r_W) \\
\ &= \ \frac{1}{2}v_8(\tw(\pi(L))-\chi(F)-\chi(\bdy Y)). \qedhere
\end{align*}
\end{proof}

Again the restriction on representativity is the best possible for \refthm{volguts}: The knot $10_{161}$, shown in \reffig{LowVolEg}, has a generalised alternating projection onto a Heegaard torus, with  $r(\pi(L),F)=4$, with 10 twist regions and 11 crossings, but $\vol(10_{161})\approx 5.6388< 5v_8$. In this case, there exists an $I$-bundle which does not come from a twist region.

\begin{figure}
\includegraphics[scale=0.8]{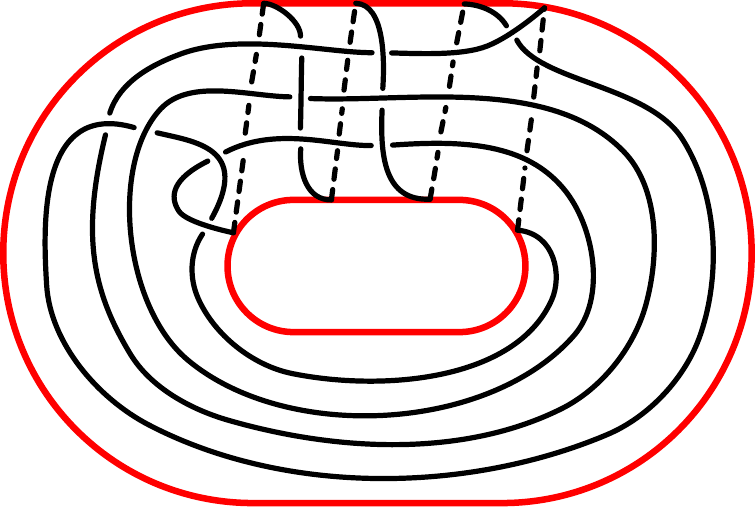}
  \caption{A generalised alternating diagram of $10_{161}$}
  \label{Fig:LowVolEg}
\end{figure}

\section{Exceptional fillings}\label{Sec:Filling}
This section combines results of the previous sections to address questions on the geometry of Dehn fillings of weakly generalised alternating links. These questions can be addressed in two ways, one geometric and the other combinatorial. Both arguments have advantages. We compare results in \refrem{CompareFilling}. 

\subsection{Geometry and slope lengths}

Let $M$ be a manifold with torus boundary whose interior admits a hyperbolic structure. Recall that a \emph{slope} on $\bdy M$ is an isotopy class of essential simple closed curves. A slope inherits a length from the hyperbolic metric, by measuring a geodesic representative on the (Euclidean) boundary of an embedded horoball neighborhood of all cusps. The 6--Theorem \cite{ago00, lac00} implies that if a slope has length greater than six, Dehn filling along that slope will result in a hyperbolic manifold. We will apply the 6--Theorem to bound exceptional fillings. First, we need to bound the lengths of slopes on weakly generalised alternating links, and this can be done by applying results of Burton and Kalfagianni~\cite{bk17}. Burton and Kalfagianni give bounds on slope lengths for hyperbolic knots that admit a pair of essential surfaces. By \refthm{hress}, weakly generalised alternating links admit such surfaces. Moreover, \refthm{Hyperbolic} gives conditions that guarantee the link complement is hyperbolic. Putting these together, we obtain the following.

\begin{theorem}\label{Thm:SlopeLengths}
Let $\pi(K)$ be a weakly generalised alternating diagram of a knot $K$ on a generalised projection surface $F$ in a closed 3-manifold $Y$ such that $Y\setminus N(F)$ is atoroidal, and such that the hypotheses of \refthm{Hyperbolic} are satisfied. Let $\mu$ denote the meridian slope on $\bdy N(K)$, i.e.\ the slope bounding a disk in $Y$, and let $\lambda$ denote the shortest slope with intersection number $1$ with $\mu$. Let $c$ be the number of crossings of $\pi(K)$. Then the lengths $\ell(\mu)$ and $\ell(\lambda)$ satisfy:
\[ \ell(\mu) \leq 3 - \frac{3\chi(F)}{c}, \quad\mbox{and}\quad \ell(\lambda) \leq 3(c-\chi(F)). \]
\end{theorem}

\begin{proof}
By \refthm{Hyperbolic}, the interior of $Y\setminus K$ is hyperbolic, and by \refthm{hress}, the checkerboard surfaces $\Sigma$ and $\Sigma'$ are $\pi_1$-essential. Let $\chi$ denote the sum $|\chi(\Sigma)| + |\chi(\Sigma')|$. Then \cite[Theorem~4.1]{bk17} implies the slope lengths of $\mu$ and $\lambda$ are bounded by:
\[ \ell(\mu) \leq\frac{6\chi}{i(\bdy\Sigma, \bdy\Sigma')}, \quad\mbox{and}\quad
\ell(\lambda) \leq 3\chi, \]
where $i(\bdy\Sigma, \bdy\Sigma')$ denotes the intersection number of $\bdy\Sigma$ and $\bdy\Sigma'$ on $\bdy N(L)$. (Note \cite[Theorem~4.1]{bk17} is stated for knots in $S^3$, but the proof applies more broadly to our situation.)

Because $\Sigma$ and $\Sigma'$ are checkerboard surfaces and $\pi(K)$ is alternating, $i(\bdy\Sigma, \bdy\Sigma')=2c$. Moreover, the sum of their Euler characteristics can be shown to be $\chi(F)-c$. 
The result follows by plugging in these values of $\chi$ and $i(\bdy\Sigma,\bdy\Sigma')$. 
\end{proof}

The results of \refthm{SlopeLengths} should be compared to those of Adams \emph{et al} in \cite{acf06}: for regular alternating knots in $S^3$, they found that $\ell(\mu) \leq 3-6/c$ and $\ell(\lambda) \leq 3c-6$, matching our results exactly when $F=S^2$.

The bounds on slope length lead immediately to bounds on exceptional Dehn fillings, and to bounds on volume change under Dehn filling.

\begin{corollary}\label{Cor:ExceptionalLengthGeom}
Suppose $\pi(K)$, $F$, and $Y$ satisfy the hypotheses of \refthm{SlopeLengths}. Let $\sigma = p\mu + q\lambda$ be a slope on $\bdy N(K)$. If $|q|>5.373(1-\chi(F)/c)$ then the Dehn filling of $X(K)$ along slope $\sigma$ yields a hyperbolic manifold. 
\end{corollary}

\begin{proof}
As in \cite{bk17}, the Euclidean area of a maximal cusp $C$ in a one-cusped $3$-manifold $X$ satisfies
\[ \area(C) \leq \frac{\ell(\sigma)\ell(\mu)}{|i(\sigma,\mu)|}. \]
Work of Cao and Meyerhoff \cite{cm01} implies $\area(C)\geq 3.35$. Plugging in the bound on $\ell(\mu)$ from \refthm{SlopeLengths}, along with the fact that $|i(\sigma, \mu)| = |q|$, and solving for $\ell(\sigma)$, we obtain
\[ \ell(\sigma) \geq \frac{3.35c}{3(c-\chi(F))}\,|q|. \]
Thus if $|q|>(18/3.35)(1-\chi(F)/c) > 5.373(1-\chi(F)/c)$, then $\ell(\sigma)>6$ and the 6--Theorem gives the result. 
\end{proof}

\begin{corollary}\label{Cor:VolumeDF}
Suppose $\pi(K)$, $F$, and $Y$ satisfy the hypotheses of \refthm{SlopeLengths}, and also
$Y$ is either closed or has only torus boundary components, and 
$r(\pi(K),F)>4$. Let $\sigma = p\mu + q\lambda$ be a slope on $\bdy N(K)$. If $|q|>5.6267(1-\chi(F)/c)$ then the Dehn filling of $X(K)$ along slope $\sigma$ yields a hyperbolic manifold $N$ with volume bounded by
\begin{gather*}
\vol(Y\setminus K) \geq \vol(N) \geq \left( 1-\left(\frac{3.35\,c}{3(c-\chi(F))}\right)^2\right)^{3/2}\vol(Y\setminus K) \\
 \geq \left( 1-\left(\frac{3.35\,c}{3(c-\chi(F))}\right)^2\right)^{3/2} \frac{v_8}{2}(\tw(\pi(K))-\chi(F)).
\end{gather*}
\end{corollary}

\begin{proof}
If $|q|>5.6267(1-\chi(F)/c)$, then $\ell(\sigma)>2\pi$ and the hypotheses of \cite[Theorem~1.1]{fkp08} hold. The lower bound on volume comes from that theorem, and from \refthm{volguts}. 
\end{proof}

\subsection{Combinatorial lengths}

We now turn to combinatorial results on Dehn filling. In \cite{lac00}, Lackenby gives conditions that guarantee that an alternating link on $S^2$ in $S^3$ admits no exceptional Dehn fillings, using the dual of angled polyhedra. We reinterpret his arguments in terms of angled polyhedra, as in \cite{fp07}, and generalise to angled chunks. The main theorem is the following.

\begin{theorem}\label{Thm:DehnFilling}
Let $\pi(L)$ be a weakly twist-reduced, weakly generalised alternating diagram of a link $L$ on a generalised projection surface $F$ in a 3-manifold $Y$.
Suppose 
$Y\setminus N(F)$ is atoroidal and $\bdy$-anannular, and $F$ has genus at least $1$. 
Finally, suppose that the regions of $F\setminus\pi(L)$ are disks, and $r(\pi(L),F)>4$.
For some component(s) $K_i$ of $L$, pick surgery coefficient $p_i/q_i$ (in lowest terms) satisfying $|q_i|> 8/\tw(K_i,\pi(L))$. Then the manifold obtained by Dehn filling $Y\setminus L$ along the slope(s) $p_i/q_i$ is hyperbolic.
\end{theorem}

Here $\tw(K,\pi(L))$ denotes the number of twist regions that the component $K$ runs through. See \cite{im16} for the complete classification of exceptional fillings on hyperbolic alternating knots when $F=S^2$ and $Y=S^3$. \refthm{DehnFilling} has the following corollaries.

\begin{corollary}\label{Cor:DehnBoundOnq}
Suppose $\pi(L)$, $F$, and $Y$ satisfy the hypotheses of \refthm{DehnFilling}.
For some component(s) $K_i$ of $L$, pick surgery coefficient $p_i/q_i$ (in lowest terms) satisfying $|q_i|> 4$. Then the manifold obtained by Dehn filling $Y\setminus L$ along the slope(s) $p_i/q_i$ is hyperbolic.
\end{corollary}

\begin{proof}
Suppose that some component $K_i$ of $L$ only passes through one twist-region. Then there exists a curve $\gamma$ on $F$ which is parallel to $K_i$ and meets $\pi(L)$ at most once. If $\gamma$ meets $\pi(L)$ once, then $\pi(L)$ is not checkerboard colourable. If $\gamma$ is essential in $F$ and does not meet $\pi(L)$, then some region of $F\setminus \pi(L)$ is not a disk. If $\gamma$ is trivial in $F$ and does not meet $\pi(L)$, then $\gamma$ can be isotoped across the twist region to meet $\pi(L)$ exactly twice, contradicting the fact that $\pi(L)$ is weakly prime. Hence $\tw(K_i,\pi(L))\geq 2$, and then by \refthm{DehnFilling}, any filling along $K_i$ where $|q_i|> 4$ will be hyperbolic.
\end{proof}

\begin{corollary}\label{Cor:DehnGenusBound}
Let $\pi(K)$ be a weakly twist-reduced, weakly generalised alternating diagram of a knot $K$ on a generalised projection surface $F$ in a 3-manifold $Y$.
Suppose $Y\setminus N(F)$ is atoroidal and $\bdy$-anannular, and $F$ has genus at least $5$. 
Finally, suppose that the regions of $F\setminus\pi(K)$ are disks, and $r(\pi(K),F)>4$. Then all non-trivial fillings of $Y\setminus K$ are hyperbolic.
\end{corollary}

\begin{proof}
Let $\Gamma'$ be a graph obtained by replacing each twist region of $\pi(K)$ by a vertex. If $t$ is the number of twist-regions in $\pi(K)$ and $f$ is the number of non-bigon faces, then Euler characteristic gives $-t+f=2-2g(F)$. Since $\pi(K)$ is checkerboard-colourable, it follows that $f\geq 2$ and hence that $t\geq 2g(F)\geq 10$. But then \refthm{DehnFilling} shows that any filling along $K$ where $q\not=0$ will be hyperbolic.
\end{proof}

\begin{remark}\label{Rem:CompareFilling}
Compare \refcor{ExceptionalLengthGeom} to \refthm{DehnFilling}. The hypotheses are slightly weaker for the corollary, and indeed, Burton and Kalfagianni's geometric estimates on slope length \cite{bk17} apply anytime a knot is hyperbolic with two essential spanning surfaces. Thus it applies to hyperbolic adequate knots, for which $r(\pi(K),F)=2$. However, the requirement on $|q|$ is also stronger in \refcor{ExceptionalLengthGeom}: it is necessary that $|q|\geq 6$ whenever $F$ has positive genus. In contrast, \refthm{DehnFilling} requires stronger restrictions on the diagram, namely $r(\pi(L),F)>4$, but works for links and only needs $|q|>4$, and often works with lower $|q|$. For example, when the diagram has a high number of twist regions, e.g.\ $\tw(K_i,\pi(L))>8$, \refthm{DehnFilling} will apply for any $q\not=0$.
\end{remark}

To prove \refthm{DehnFilling}, we will define a combinatorial length of slopes, and show that if the combinatorial length is at least $2\pi$, then the Dehn filling is hyperbolic. First, we need to extend our definition of combinatorial area to surfaces that may not be embedded, and may have boundary components lying in the interior of a chunk. 

\begin{definition}\label{Def:Admissible}
Let $S$ be the general position image of a map of a connected compact surface into a chunk $C$. Then $S$ 
is \emph{admissible} if:
\begin{itemize}
\item If $S$ is closed, it is $\pi_1$-injective in $C$.
\item $S$ and $\bdy S$ are transverse to all faces, boundary faces, and edges of $C$.
\item $\bdy S\setminus\bdy C$ is a (possibly empty) collection of embedded arcs with endpoints in interior faces of $C$, or embedded closed curves. 
\item $\bdy S\cap \bdy C$ consists of a collection of immersed closed curves or immersed arcs.
\item If an arc $\gamma_i$ of $\bdy S$ lies entirely in a single face of $C$, then either the arc is embedded in that face, or the face is not simply connected, and every subarc of $\gamma_i$ that forms a closed curve on the face does not bound a disk on that face. 
\item Each closed curve component satisfies conditions (3) through (5) of the definition of normal, \refdef{NormalSurface}. 
\item Each arc component satisfies conditions (4) and (5) of \refdef{NormalSurface}.
\end{itemize}
A surface $S'$ in an irreducible, $\bdy$-irreducible 3-manifold $M$ with a chunk decomposition is \emph{admissible} if each component of intersection of $S'$ with a chunk is admissible. 
\end{definition}

We noted above that admissible surfaces do not need to be embedded. In fact, they do not even need to be immersed. For example, an admissible surface could be the image of a compact surface under a continuous map. If $S$ is any immersed surface that is $\pi_1$-injective and boundary $\pi_1$-injective in a 3-manifold $M$ that is irreducible and $\bdy$-irreducible, we may homotope $S$ to be admissible. A similar result holds more generally for images of compact surfaces in general position.

We now adjust the definition of combinatorial area to include admissible surfaces.

\begin{definition}\label{Def:CombAreaAdmissible}
  Let $S$ be an admissible surface in an angled chunk $C$. Let $\sigma(\bdy S\setminus\bdy C)$ denote the number of components of the intersection of $\bdy S$ with the interior of $C$. For each point $v$ on $\bdy S$ lying between an interior face of $C$ and an arc of $\bdy S\setminus\bdy C$, define the exterior angle $\epsilon(v) = \pi/2$. Let $v_1, \dots, v_n \in \bdy S$ be all the points where $\bdy S$ meets edges of $C$, and all points on $\bdy S$ between an interior face of $C$ and an arc of $\bdy S\setminus\bdy C$. The \emph{combinatorial area} of $S$ is
  \[ a(S) = \sum_{i=1}^n \epsilon(v_i) - 2\pi\chi(S) + 2\pi\sigma(\bdy S\setminus\bdy C).\]

  Let $S'$ be an admissible surface in a 3-manifold $M$ with an angled chunk decomposition. Then the combinatorial area of $S'$ is the sum of the combinatorial areas of its components of intersection with the chunks. 
\end{definition}

\begin{proposition}[Gauss--Bonnet for admissible surfaces]\label{Prop:GaussBonnetAdmissible}
  Let $S$ be an admissible surface in an angled chunk decomposition. Let $\sigma(\bdy S \setminus \bdy M)$ be the number of arcs of intersection of $\bdy S\setminus \bdy M$ and the chunks. Then
  \[ a(S) = -2\pi\chi(S) + 2\pi\sigma(\bdy S\setminus\bdy M). \]
\end{proposition}

\begin{proof}
The proof is by induction, identical to that of \refprop{GaussBonnet}, except keeping track of a term $\sigma$ at each step, which remains unchanged throughout the proof. 
\end{proof}

\begin{definition}\label{Def:CombLength}
Let $C$ be an angled chunk in a chunk decomposition of $M$, and let $S$ be an admissible surface in $C$ that intersects at least one boundary face. Let $\gamma$ be an arc of intersection of $S$ with a boundary face of $C$. Define the \emph{length of $\gamma$ relative to $S$} to be
  \[ \ell(\gamma,S) = \frac{a(S)}{|\bdy S\cap \bdy M|}. \]

Suppose $\gamma$ is an immersed arc in a component of $\bdy M$ meeting boundary faces, with $\gamma$ disjoint from vertices on $\bdy M$. Let $\gamma_1, \dots, \gamma_n$ denote the arcs of intersection of $\gamma$ with boundary faces. Suppose that each $\gamma_i$ is embedded and no $\gamma_i$ has endpoints on the same boundary edge. For each $i$, let $S_i$ be an admissible surface in the corresponding chunk whose boundary contains $\gamma_i$. Then $S=\cup_{i=1}^n S_i$ is an \emph{inward extension of $\gamma$} if $\bdy S_i$ agrees with $\bdy S_{i+1}$ on their shared interior face and if $\gamma$ is closed, $\bdy S_n$ agrees with $\bdy S_1$ on their shared face.

Define the \emph{combinatorial length} of $\gamma$ to be
\[ \ell_c(\gamma) = \inf \left\{ \sum_{i=1}^n \ell(\gamma_i, S_i) \right\}, \]
where the infimum is taken over all inward extensions of $\gamma$. 
\end{definition}

The following is a generalisation of \cite[Proposition~4.8]{lac00}. 
\begin{proposition}\label{Prop:AdmissibleSlopeLength}
Let $S$ be an admissible surface in an angled chunk decomposition on $M$. Let $\gamma_1, \dots, \gamma_m$ be the components of $\bdy S\cap \bdy M$ on torus components of $\bdy M$ made up of boundary faces, each $\gamma_j$ representing a non-zero multiple of some slope $s_{i_j}$. Then
  \[ a(S) \geq \sum_{j=1}^m \ell_c(s_{i_j}). \]
\end{proposition}

\begin{proof}
  Intersections of $S$ with the chunks form an inward extension of $\gamma_j$. Sum the lengths of arcs relative to these intersections. 
\end{proof}

Propositions~\ref{Prop:GaussBonnetAdmissible} and~\ref{Prop:AdmissibleSlopeLength} allow us to generalise the following result of Lackenby. 

\begin{theorem}[Combinatorial $2\pi$-theorem, \cite{lac00}]\label{Thm:Comb2PiThm}
Let $M$ be a compact orientable 3-manifold with an angled chunk decomposition. Suppose $M$ is atoroidal and not Seifert fibered and has boundary
containing
a non-empty union of tori. Let $s_1, \dots, s_n$ be a collection of slopes on $\bdy M$, at most one for each component made up of boundary faces (but no slopes on exterior faces that are tori). Suppose for each $i$, $\ell_c(s_i)>2\pi$. Then the manifold $M(s_1, \dots, s_n)$ obtained by Dehn filling $M$ along these slopes is hyperbolic. 
\end{theorem}

\begin{proof}
The proof is nearly identical to that of \cite[Theorem~4.9]{lac00}. In particular, if $M(s_1, \dots, s_n)$ is toroidal, annular, reducible, or $\bdy$-reducible, then it contains an essential punctured torus, annulus, sphere, or disk $S$ with punctures on slopes $s_i$. The punctured surface $S$ can be put into normal form. \refprop{GaussBonnetAdmissible} implies $a(S) = -2\pi\chi(S) \leq 2\pi|\bdy S|$, and \refprop{AdmissibleSlopeLength} implies $a(S)>2\pi|\bdy S|$, a contradiction.

If not all components of $\bdy M$ are filled, then Thurston's hyperbolisation theorem
immediately implies $M(s_1, \dots, s_n)$ is hyperbolic. If $M(s_1, \dots, s_n)$ is closed, hyperbolicity will follow when we show it has word hyperbolic fundamental group. We use Gabai's ubiquity theorem \cite{gab98} in the form stated as \cite[Theorem~2.1]{lac00}. 

Suppose $M(s_1, \dots, s_n)$ is closed and the core of a surgery solid torus has finite order in $\pi_1(M(s_1, \dots, s_n))$. Then there exists a singular disk with boundary on the core. Putting the disk into general position with respect to all cores and then drilling, we obtain a punctured singular disk $S$ in $M$ with all boundary components on multiples of the slopes $s_i$. If necessary, replace $S$ with a $\pi_1$-injective and boundary $\pi_1$-injective surface with the same property. Because the boundary of $S$ meets $\bdy M$ in boundary faces, and runs monotonically through these faces, we may homotope $S$ so that each arc of $\bdy S$ in a boundary face is embedded (note this is not necessarily possible if $\bdy S$ lies on an exterior face). Similarly, moving singular points away from edges and faces of the chunk decomposition, then using the fact that $M$ is irreducible and boundary irreducible, we may homotope $S$ to satisfy the other requirements of an admissible surface. Then again \refprop{GaussBonnetAdmissible} implies $a(S)=-2\pi\chi(S) \leq 2\pi|\bdy S|$ and \refprop{AdmissibleSlopeLength} implies $a(S)>2\pi|\bdy S|$, a contradiction. Thus each core of a surgery solid torus has infinite order.

Let $\gamma$ be a curve in $M$ that is homotopically trivial in $M(s_1, \dots, s_n)$, and arrange $\gamma$ to meet the chunk decomposition transversely. Let $S$ be a compact planar surface in $M$ with $\bdy S$ consisting of nonzero multiples of the slopes $s_i$ as well as $\gamma$. If necessary, replace $S$ with a $\pi_1$-injective and boundary $\pi_1$-injective surface with the same property, and adjust $S$ to be admissible. Let $\sigma(\gamma)$ be the number of arcs of intersection between $\gamma$ and the chunks of $M$. Let $\epsilon>0$ be such that $\ell_c(s_i)\geq 2\pi+\epsilon$ for all $i$. Then
\begin{align*}
(2\pi+\epsilon)|S\cap \bdy M| & \leq \sum_{j=1}^{|S\cap \bdy M|} \ell_c (s_{i_j}) \hspace{1.1in} \mbox{by definition of combinatorial length}\\
& \leq a(S) = -2\pi\chi(S) + 2\pi\sigma(\gamma) \quad \mbox{ by \refprop{AdmissibleSlopeLength} } \\
& < 2\pi|S\cap \bdy M| + 2\pi\sigma(\gamma) \quad\quad\quad\: \mbox{ since } \chi(S)=2-(|S\cap\bdy M|+1), \mbox{ so}
\end{align*}
\[ |S\cap \bdy M| < (2\pi/\epsilon)\sigma(\gamma). \]

We have found a constant $c=2\pi/\epsilon$ that depends on $M$ and $s_1, \dots, s_n$, but is independent of $\gamma$ and $S$ with $|S\cap \bdy M|\leq c\sigma(\gamma)$. Since $\sigma(\gamma)$ is the length of $\gamma$ in a simplicial metric on the chunk decomposition, the ubiquity theorem \cite[Theorem~2.1]{lac00} implies $\pi_1(M(s_1, \dots, s_n))$ is word hyperbolic. 
\end{proof}

\begin{theorem}\label{Thm:CombLengthWGA}
Let $\pi(L)$ be a weakly twist-reduced weakly generalised alternating diagram of a link $L$ on a generalised projection surface $F$ in a 3-manifold $Y$. Suppose $Y\setminus N(F)$ is atoroidal and $\bdy$-anannular, $F$ has genus at least one, all regions of $F\setminus\pi(L)$ are disks, and $r(\pi(L),F)>4$. Then the combinatorial length of the slope $p/q$ on a component $K$ of $L$ is at least $|q|\tw(K,\pi(L))\pi/4$. 
\end{theorem}

\begin{proof}[Proof of \refthm{DehnFilling} from \refthm{CombLengthWGA}]
By \refthm{Hyperbolic}, $Y\setminus L$ is hyperbolic. Thus it is atoroidal and not Seifert fibered. By \refthm{CombLengthWGA}, the combinatorial length of $p/q$ is at least $|q|\tw(K,\pi(L))\pi/4$. By hypothesis, $|q|>8/\tw(K,\pi(L))$, so $\ell_c(p/q)>2\pi$.
Since $K$ lies on $F$, in the angled chunk decomposition of the link exterior, each slope on $K$ lies on a torus boundary component made up of boundary faces and not exterior faces. 
Then \refthm{Comb2PiThm} implies that the manifold obtained by Dehn filling $K$ along slope $p/q$ is hyperbolic. 
\end{proof}

The remainder of the section is devoted to the proof of \refthm{CombLengthWGA}.

\begin{lemma}[See Lemmas~4.2 and~5.9 of \cite{lac00}]\label{Lem:Marc4.2}
Let $\pi(L)$ be a weakly twist-reduced, weakly generalised alternating diagram on a generalised projection surface $F$ in a 3-manifold $Y$. 
Suppose all regions of $F\setminus\pi(L)$ are disks and $r(\pi(L),F)>4$. If $S$ is an orientable admissible surface in a chunk $C$ of $X:=Y\setminus N(L)$, and $S\cap \bdy X\neq \emptyset$, and further $\bdy S$ does not agree with the boundary of a normal surface, then $a(S)>0$.

Furthermore, if $S$ is normal or admissible and $a(S)>0$, then $a(S)/|S\cap \bdy X| \geq \pi/4$.
\end{lemma}

\begin{proof}
Suppose $a(S)\leq 0$. By definition, $a(S) = \sum \epsilon(v_i) - 2\pi\chi(S) + 2\pi\sigma(\bdy S\setminus \bdy X)$. 

If $\chi(S)<0$, then the combinatorial area is positive. So $\chi(S)\geq 0$. Since $S$ meets a boundary face, it cannot be a sphere or torus. It follows that $S$ is a disk or annulus. If an annulus, then $a(S) \geq \sum\epsilon(v_i)$, and since $S\cap \bdy X\neq \emptyset$, at least two terms in the sum come from intersections with boundary edges, so the sum is strictly positive. So $S$ is a disk.

Next, note if $\sigma(\bdy S\setminus \bdy X)>1$, then $a(S)>0$. If $\sigma(\bdy S\setminus \bdy X)=1$ and $a(S)\leq 0$, then $\sum \epsilon(v_i)=0$. Again this is impossible: two terms in the sum arise as endpoints of an arc in $\bdy S\setminus \bdy X$, and each has angle $\epsilon(v_i)=\pi/2$. So $\sigma(\bdy S\setminus \bdy X)=0$.

Now as in the proof of \cite[Lemma~4.2]{lac00}, consider the number of intersections of $S$ with interior edges of the chunk. If $S$ meets no interior edges, then it meets at least two boundary faces
(by representativity and weakly prime conditions), so $a(S)\geq 0$.
If $\bdy S$ is embedded, the Loop Theorem gives an embedded disk with the same boundary, but by assumption, $\bdy S$ is not the boundary of a normal disk. Thus there is a point $p$ of intersection of $\bdy S$ with itself. 
Then there is an arc $\alpha$ of $\bdy S$ running from $p$ to $p$ that meets $\bdy X$ fewer times than $\bdy S$, and a disk with boundary $\alpha$. If the boundary of the disk is not embedded, repeat. If it is embedded, the Loop Theorem gives a normal disk with area strictly less than that of $S$. \refprop{NonnegArea} then implies $a(S)>0$.

So now suppose $\bdy S$ meets an interior edge. Again $\bdy S$ is not embedded. 
Then there is a point of self intersection $p$ in $\bdy S$. Let $\alpha$ be an arc of $\bdy S$ running from $p$ to $p$; we can ensure it meets fewer edges or boundary edges than $\bdy S$. There exists a disk with boundary $\alpha$. If the boundary of $\alpha$ is not embedded, we repeat until we have a disk with embedded boundary. By the Loop Theorem, we obtain a normal disk $D$ meeting strictly fewer edges and boundary edges than $\bdy S$. \refprop{NonnegArea} implies $a(D)\geq 0$. So $a(S)>0$.

Finally, we show if $a(S)>0$ then $a(S)/|S\cap \bdy X| \geq \pi/4$. By definition,
\[ \frac{a(S)}{|S\cap \bdy X|} = \frac{\sum \epsilon(v_i) - 2\pi\chi(S) + 2\pi\sigma(\bdy S\setminus \bdy X)}{|S\cap\bdy X|}. \]
The sum $\sum \epsilon(v_i)$ breaks into $\pi|S\cap \bdy X| + \pi\sigma(\bdy S\setminus \bdy X) + \sum_{j} \epsilon(w_j)$ where $\{w_1,\dots,w_m\}$ denote the intersections of $\bdy S$ with interior edges, and each $\epsilon(w_j)=\pi/2$.

Note $\chi(S)\leq 1$ because $S$ is not a sphere. Also, $\sigma(\bdy S\setminus \bdy X)\geq 0$. Hence
\[ \frac{a(S)}{|S\cap \bdy X|} \geq \frac{\sum_{w_j} \pi/2}{|S\cap\bdy X|} + \pi - \frac {2\pi}{|S\cap\bdy X|}. \]

If $|S\cap \bdy X|\geq 3$, $a(S)/|S\cap\bdy X| \geq 0 + \pi -2\pi/3 >\pi/4.$

If $|S\cap \bdy X|=2$, $a(S)/|S\cap\bdy X| \geq (\sum_{w_j} \pi/2)/2 + \pi - (2\pi)/2 = m\pi/4$. Since $a(S)>0$, it follows that $m\geq 1$ and the sum is at least $\pi/4$. 

If $|S\cap \bdy X|=1$, $a(S)/|S\cap\bdy X| \geq \sum_{w_j} \pi/2 + \pi - 2\pi = m\pi/2 -\pi.$ Since $a(S)>0$, $\bdy S$ meets at least three interior edges, so this is at least $\pi/2 >\pi/4$. 
\end{proof}

\begin{remark}
  The proof of \cite[Lemma~4.2]{lac00} had additional cases, because the definition of normal in that paper required $\bdy S$ to avoid intersecting the same boundary face more than once, and the same interior edge more than once. Because our definition of normal, based off that of \cite{fg09}, did not have these restrictions, our proof is simpler.
\end{remark}

We now consider the boundary faces of an angled chunk. These are all squares. One diagonal of the square, running between vertices of the square that correspond to identified edges, forms a meridian of the link. The squares glue together to form what Adams calls a \emph{harlequin tiling} \cite{acf06}; see \reffig{HarlequinTiling}.
Unwind the boundary $\bdy X$ in the longitude direction by taking the cover corresponding to a meridian. Assign an $x$-coordinate to the centre of each square, with adjacent coordinates of squares differing by $1$, as on the right of \reffig{HarlequinTiling}. 

\begin{figure}
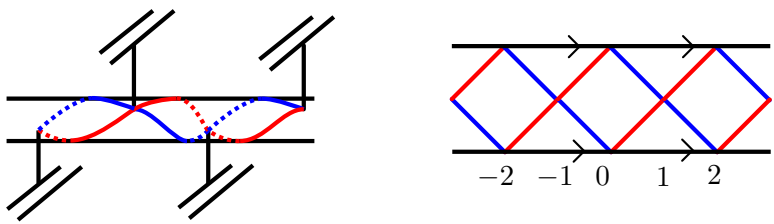
\caption{The boundary tiling of an alternating knot is called a harlequin tiling. Right: unwinding in the longitude direction gives a string of quads as shown, with coordinates in $\ZZ$.}
\label{Fig:HarlequinTiling}
\end{figure}

A curve $\gamma$ representing a non-zero multiple $k$ of the slope $p/q$ lifts to an arc $\widetilde{\gamma}$ starting at $x=0$, ending at $k|q|\Cr(K,\pi(L))$, where $\Cr(K,\pi(L))$ denotes the number of crossings met by $K$. Note that crossings of $\pi(L)$ are double-counted if both strands are part of $K$.
Let $\gamma_i$ denote the $i$-th intersection of $\widetilde{\gamma}$ with lifts of boundary squares. The arc $\gamma_i$ lies in one boundary square, and has some $x$-coordinate $x_i$. Define $x_i'$ to be the integer $x_i' = (x_{i-1}+x_{i+1})/2$. Say $\gamma_i$ is a \emph{skirting arc} if it has endpoints on adjacent edges of a boundary square. 

Now let $S$ be an inward extension of $\gamma$. Each $\gamma_i$ is then part of the boundary of an admissible surface $S_i$ lying in a single chunk.

\begin{lemma}[See Lemmas~5.5, 5.6, 5.7, and 5.8 of \cite{lac00}]\label{Lem:Marc5.5-5.6-5.7-5.8}
The following hold for $S_i$, $\gamma_i$:
  \begin{enumerate}
  \item[(5.6)] If $S_i$ is admissible and $\gamma_i$ is a skirting arc, either $a(S_i)>0$ or $S_i$ is a boundary bigon.
  \item[(5.7)] Let $\gamma_i$ and $\gamma_{i+1}$ be non-skirting arcs in boundary squares $B_i$ and $B_{i+1}$, such that the gluing map gluing a side of $S_i$ to a side of $S_{i+1}$ does not take the point on $B_i$ to that on $B_{i+1}$ by rotating (clockwise or counter-clockwise) past a bigon edge. Then at least one of $a(S_i)$, $a(S_{i+1})$ is positive.
  \item[(5.8)] If exactly one of $\gamma_i$, $\gamma_{i+1}$ is a skirting arc, and $B_i$, $B_{i+1}$ satisfy the same hypothesis as in (5.7) above, and $x_i'\neq x_{i+1}'$, then at least one of $a(S_i)$, $a(S_{i+1})$ is positive.
  \item[(5.5)] There are at least $k|q|$ arcs $\gamma_i$ with $x_i'\neq x_{i+1}'$.
  \end{enumerate}
\end{lemma}

\begin{proof}
  For (5.6), (5.7), and (5.8) we will assume the combinatorial areas are zero. Then \reflem{Marc4.2} implies the $S_i$ are normal, and \reflem{NormalSquare} implies it has one of three forms: form (1) meeting a boundary face and two interior edges, form (2) a boundary bigon, or form (3) meeting two boundary faces and interior faces of the same colour.
  
  (5.6) If $\gamma_i$ is a skirting arc, $a(S_i)=0$, and $S_i$ is not a boundary bigon, it has form (1) or (3). But both of these give non-skirting arcs.

  (5.7) If $a(S_i)=a(S_{i+1})=0$, then both are normal disks of form (1) or (3) of \reflem{NormalSquare} because the arcs are non-skirting. Then $\bdy S_i$ and $\bdy S_{i+1}$ glue in (without loss of generality) a white face of the chunk via a clockwise twist. As in \refsec{Hyperbolic}, superimpose $\bdy S_i$ and $\bdy S_{i+1}$ onto the boundary of the same chunk.
In the first case, $\bdy S_i \cap \bdy S_{i+1}$ must be nonempty. Then \reflem{MarcLemma7} implies the boundaries of the squares meet in opposite white faces. Since the diagram is weakly prime, this forces $\bdy S_{i+1}$ and $\bdy S_i$ to bound a bigon, as on the left of \reffig{5.7}, contradicting the assumption that there is no bigon edge adjacent to the boundary squares. In the case $S_i$ and $S_{i+1}$ are of form (3), weakly twist-reduced implies that both bound a string of bigons, and the two have just one edge between them implies that there is a single bigon between $\bdy S_i$ and $\bdy S_{i+1}$, as on the right of \reffig{5.7}. Again this is a contradiction.

\begin{figure}
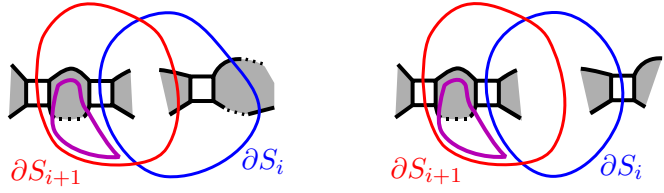
\caption{Left: if normal disks $S_i$ and $S_{i+1}$ are of form (1), they bound a bigon between them. Right: the same conclusion holds if they are of form (3).}
\label{Fig:5.7}
\end{figure}

  (5.8) If $a(S_i)=a(S_{i+1})=0$, either $S_i$ is a boundary bigon and $S_{i+1}$ is of form (3) or vice versa. They glue in a face via a twist. As in the previous argument, without loss of generality $S_i$ glues to $S_{i+1}$ in a white face by a single clockwise rotation. If $S_i$ is a boundary bigon, then the gluing implies $S_{i+1}$ bounds a bigon, and this contradicts our assumption on $B_i$, $B_{i+1}$. So $S_{i+1}$ is a boundary bigon and $S_i$ is of form (3). See \reffig{5.8}, left and middle.
In this case, $S_{i+1}$ encircles an edge adjacent to a bigon region bounded by $S_i$. The arc $\gamma_{i+1}$ is a skirting arc running vertically in the tiling of \reffig{HarlequinTiling} (vertical because it must cut off the edge meeting the boundary quad from the upper chunk twice; see \reffig{5.8} right), so
  \[ x_i' = \frac{x_i +1 + x_i - 1}{2} = x_i \quad \mbox{and} \quad
  x_{i+1}' = \frac{x_i + x_i}{2} = x_i. \]
  This contradicts the assumption that $x_i\neq x_{i+1}'$.

\begin{figure}
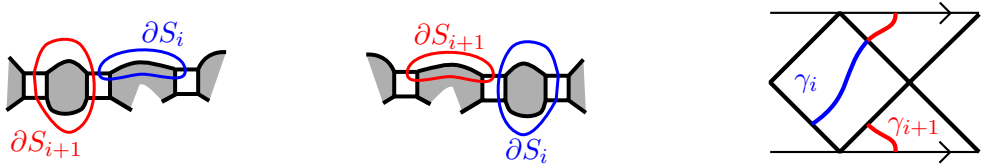
\caption{Left: If $S_{i+1}$ is of form (3) and $S_i$ is a boundary bigon. Middle: $S_{i+1}$ is a boundary bigon and $S_i$ is of form (3). Right: $\gamma_{i+1}$ is a skirting arc running vertically.}
\label{Fig:5.8}
\end{figure}

  (5.5)
  As in \cite[Lemma~5.5]{lac00}: $|x_{i+1}' - x_i'|\leq 1$, and if $x_i'\neq x_{i+1}'$ then $\{ x_i',x_{i+1}'\} = \{x_i, x_{i+1}\}$. Thus for each $j\in \ZZ$ with $x_1'\leq j\leq x_1'+k|q|\Cr(K, \pi(L))$, the pair $\{j, j+1\}$ occurs as $\{x_i', x_{i+1}'\}$ for some arc. 
\end{proof}

\begin{proof}[Proof of \refthm{CombLengthWGA}]
  Let $\gamma$ be a curve representing a nonzero multiple $k$ of the slope $p/q$, and let $S$ be an inward extension, with $S_i$ the intersections of $S$ with chunks as above.

  There are $2\tw(K,\pi(L))$ boundary faces meeting $\gamma$ that are not adjacent to their right neighbour across a bigon region of the chunk. The previous lemma implies there are $k|q|$ arcs on $\bdy X$ with $x_i'\neq x_{i+1}'$. Let $\gamma_i$ be one such arc. Then $S_i$, $S_{i+1}$ cannot both be boundary bigons. The previous lemma implies that at least one of $S_i$, $S_{i+1}$ has positive combinatorial area, say $S_i$. Then \reflem{Marc4.2} implies $a(S_i)/|S_i \cap \bdy X| \geq \pi/4$. Then $\ell(\gamma_i,S_i)$ satisfies
  \[ \ell(\gamma_i,S_i) \geq (1/2)\,2\,\tw(K,\pi(L))\,k|q|\,\pi/4 \geq \tw(K,\pi(L))|q|\pi/4.\]
  The factor $1/2$ ensures we have not double counted: for non-skirting $\gamma_i$ it may be the case that both $S_{i-1}$ and $S_{i+1}$ are boundary bigons, and so we must share the combinatorial area of $S_i$ between them.
  
  Now, since $\gamma$ was an arbitrary representative of a nonzero multiple of the slope $p/q$, and since $S$ was an arbitrary inward extension, it follows that
  \[\ell_c(\gamma) \geq \pi|q|\tw(K,\pi(L))/4.\qedhere \]
\end{proof}

\bibliography{hyperbolicwga}

\providecommand{\bysame}{\leavevmode\hbox to3em{\hrulefill}\thinspace}
\providecommand{\MR}{\relax\ifhmode\unskip\space\fi MR }
\providecommand{\MRhref}[2]{%
  \href{http://www.ams.org/mathscinet-getitem?mr=#1}{#2}
}
\providecommand{\href}[2]{#2}
\begin{thebibliography}{10}

\bibitem{Adams:ThickenedSfces}
C.~Adams, C.~Albors-Riera, B.~Haddock, Z.~Li, D.~Nishida, B.~Reinoso, and
  L.~Wang, \emph{Hyperbolicity of links in thickened surfaces}, Topology Appl.
  \textbf{256} (2019), 262--278. \MR{3916014}

\bibitem{acf06}
C.~Adams, A.~Colestock, J.~Fowler, W.~Gillam, and E.~Katerman, \emph{Cusp size
  bounds from singular surfaces in hyperbolic 3-manifolds}, Trans. Amer. Math.
  Soc. \textbf{358} (2006), no.~2, 727--741. \MR{2177038}

\bibitem{ada94}
Colin~C. Adams, \emph{Toroidally alternating knots and links}, Topology
  \textbf{33} (1994), no.~2, 353--369.

\bibitem{ada03}
\bysame, \emph{Hyperbolic knots}, Handbook of knot theory, Elsevier B. V.,
  Amsterdam, 2005, pp.~1--18. \MR{2179259}

\bibitem{ada07}
\bysame, \emph{Noncompact {F}uchsian and quasi-{F}uchsian surfaces in
  hyperbolic 3-manifolds}, Algebr. Geom. Topol. \textbf{7} (2007), 565--582.

\bibitem{abb92}
Colin~C. Adams, Jeffrey~F. Brock, John Bugbee, and et~al., \emph{Almost
  alternating links}, Topology Appl. \textbf{46} (1992), no.~2, 151--165.

\bibitem{acm17}
Colin~C. Adams, Aaron Calderon, and Nathaniel Mayer, \emph{Generalized
  bipyramids and hyperbolic volumes of alternating $k$-uniform tiling links},
  \url{http://arxiv.org/abs/1709.00432}, 2017.

\bibitem{ago00}
Ian Agol, \emph{Bounds on exceptional {D}ehn filling}, Geom. Topol. \textbf{4}
  (2000), 431--449. \MR{1799796}

\bibitem{ast07}
Ian Agol, Peter~A. Storm, and William~P. Thurston, \emph{Lower bounds on
  volumes of hyperbolic {H}aken 3-manifolds}, J. Amer. Math. Soc. \textbf{20}
  (2007), no.~4, 1053--1077, With an appendix by Nathan Dunfield.

\bibitem{aum56}
Robert Aumann, \emph{Asphericity of alternating knots}, Ann. of Math. (2)
  \textbf{64} (1956), 374--392.

\bibitem{bon86}
Francis Bonahon, \emph{Bouts des vari\'et\'es hyperboliques de dimension
  {$3$}}, Ann. of Math. (2) \textbf{124} (1986), no.~1, 71--158. \MR{847953}

\bibitem{bk17}
Stephan~D. Burton and Efstratia Kalfagianni, \emph{Geometric estimates from
  spanning surfaces}, Bull. Lond. Math. Soc. \textbf{49} (2017), no.~4,
  694--708.

\bibitem{ceg87}
R.~D. Canary, D.~B.~A. Epstein, and P.~Green, \emph{Notes on notes of
  {T}hurston}, Analytical and geometric aspects of hyperbolic space
  ({C}oventry/{D}urham, 1984), London Math. Soc. Lecture Note Ser., vol. 111,
  Cambridge Univ. Press, Cambridge, 1987, pp.~3--92. \MR{903850}

\bibitem{cm01}
Chun Cao and G.~Robert Meyerhoff, \emph{The orientable cusped hyperbolic
  {$3$}-manifolds of minimum volume}, Invent. Math. \textbf{146} (2001), no.~3,
  451--478. \MR{1869847}

\bibitem{ckp16}
Abhijit Champanerkar, Ilya Kofman, and Jessica~S. Purcell, \emph{Geometrically
  and diagrammatically maximal knots}, J. Lond. Math. Soc. (2) \textbf{94}
  (2016), no.~3, 883--908. \MR{3614933}

\bibitem{ckp}
\bysame, \emph{Geometry of biperiodic alternating links}, J. Lond. Math. Soc.
  \textbf{99} (2019), no.~3, 807--830.

\bibitem{dfk08}
Oliver~T. Dasbach, David Futer, Efstratia Kalfagianni, Xiao-Song Lin, and
  Neal~W. Stoltzfus, \emph{The {J}ones polynomial and graphs on surfaces}, J.
  Combin. Theory Ser. B \textbf{98} (2008), no.~2, 384--399. \MR{2389605}

\bibitem{fen98}
S\'ergio~R. Fenley, \emph{Quasi-{F}uchsian {S}eifert surfaces}, Math. Z.
  \textbf{228} (1998), no.~2, 221--227. \MR{1630563}

\bibitem{fg09}
David Futer and Fran{\c{c}}ois Gu{\'e}ritaud, \emph{Angled decompositions of
  arborescent link complements}, Proc. Lond. Math. Soc. (3) \textbf{98} (2009),
  no.~2, 325--364. \MR{2481951}

\bibitem{fkp08}
David Futer, Efstratia Kalfagianni, and Jessica~S. Purcell, \emph{Dehn filling,
  volume, and the {J}ones polynomial}, J. Differential Geom. \textbf{78}
  (2008), no.~3, 429--464. \MR{2396249}

\bibitem{fkp13}
\bysame, \emph{Guts of surfaces and the colored {J}ones polynomial}, Lecture
  Notes in Mathematics, vol. 2069, Springer, Heidelberg, 2013.

\bibitem{fkp14}
\bysame, \emph{Quasifuchsian state surfaces}, Trans. Amer. Math. Soc.
  \textbf{366} (2014), no.~8, 4323--4343.

\bibitem{fkp15}
\bysame, \emph{Hyperbolic semi-adequate links}, Comm. Anal. Geom. \textbf{23}
  (2015), no.~5, 993--1030.

\bibitem{fp07}
David Futer and Jessica~S. Purcell, \emph{Links with no exceptional surgeries},
  Comment. Math. Helv. \textbf{82} (2007), no.~3, 629--664. \MR{2314056}

\bibitem{gab98}
David Gabai, \emph{Quasi-minimal semi-{E}uclidean laminations in
  {$3$}-manifolds}, Surveys in differential geometry, {V}ol. {III}
  ({C}ambridge, {MA}, 1996), Int. Press, Boston, MA, 1998, pp.~195--242.
  \MR{1677889}

\bibitem{gre17}
Joshua~E. Greene, \emph{Alternating links and definite surfaces}, Duke Math. J.
  \textbf{166} (2017), no.~11, 2133--2151.

\bibitem{hay95}
Chuichiro Hayashi, \emph{Links with alternating diagrams on closed surfaces of
  positive genus}, Math. Proc. Cambridge Philos. Soc. \textbf{117} (1995),
  no.~1, 113--128.

\bibitem{how15t}
Joshua~A. Howie, \emph{Surface-alternating knots and links}, Ph.D. thesis,
  University of Melbourne, 2015.

\bibitem{how17}
\bysame, \emph{A characterisation of alternating knot exteriors}, Geom. Topol.
  \textbf{21} (2017), no.~4, 2353--2371. \MR{3654110}

\bibitem{hr16}
Joshua~A. Howie and J.~Hyam Rubinstein, \emph{Weakly generalised alternating
  links}, in preparation, 2019.

\bibitem{im16}
Kazuhiro Ichihara and Hidetoshi Masai, \emph{Exceptional surgeries on
  alternating knots}, Comm. Anal. Geom. \textbf{24} (2016), no.~2, 337--377.
  \MR{3514563}

\bibitem{lac00}
Marc Lackenby, \emph{Word hyperbolic {D}ehn surgery}, Invent. Math.
  \textbf{140} (2000), no.~2, 243--282.

\bibitem{lac04}
\bysame, \emph{The volume of hyperbolic alternating link complements}, Proc.
  London Math. Soc. (3) \textbf{88} (2004), no.~1, 204--224, With an appendix
  by Ian Agol and Dylan Thurston.

\bibitem{lp16}
Marc Lackenby and Jessica Purcell, \emph{Cusp volumes of alternating knots},
  Geom. Topol. \textbf{20} (2016), no.~4, 2053--2078. \MR{3548463}

\bibitem{men83}
William~W. Menasco, \emph{Polyhedra representation of link complements},
  Low-dimensional topology ({S}an {F}rancisco, {C}alif., 1981), Contemp. Math.,
  vol.~20, Amer. Math. Soc., Providence, RI, 1983, pp.~305--325. \MR{718149}

\bibitem{men84}
\bysame, \emph{Closed incompressible surfaces in alternating knot and link
  complements}, Topology \textbf{23} (1984), no.~1, 37--44.

\bibitem{mos71}
Louise Moser, \emph{Elementary surgery along a torus knot}, Pacific J. Math.
  \textbf{38} (1971), 737--745.

\bibitem{oza06}
Makoto Ozawa, \emph{Non-triviality of generalized alternating knots}, J. Knot
  Theory Ramifications \textbf{15} (2006), no.~3, 351--360.

\bibitem{oza11}
\bysame, \emph{Essential state surfaces for knots and links}, J. Aust. Math.
  Soc. \textbf{91} (2011), no.~3, 391--404.

\bibitem{or12}
Makoto Ozawa and Joachim~Hyam Rubinstein, \emph{On the {N}euwirth conjecture
  for knots}, Comm. Anal. Geom. \textbf{20} (2012), no.~5, 1019--1060.
  \MR{3053620}

\bibitem{ot03}
Makoto Ozawa and Yukihiro Tsutsumi, \emph{Totally knotted {S}eifert surfaces
  with accidental peripherals}, Proc. Amer. Math. Soc. \textbf{131} (2003),
  no.~12, 3945--3954 (electronic).

\bibitem{tt14}
Morwen Thistlethwaite and Anastasiia Tsvietkova, \emph{An alternative approach
  to hyperbolic structures on link complements}, Algebr. Geom. Topol.
  \textbf{14} (2014), no.~3, 1307--1337.

\bibitem{thu79}
William Thurston, \emph{The geometry and topology of three-manifolds}, 1979,
  {\tt http://\allowbreak www.msri.org/\allowbreak gt3m/}.

\bibitem{thu82}
\bysame, \emph{Three-dimensional manifolds, {K}leinian groups and hyperbolic
  geometry}, Bull. Amer. Math. Soc. (N.S.) \textbf{6} (1982), no.~3, 357--381.

\end{thebibliography}
\bibliographystyle{amsplain}

\end{document}